%% file: ex_article_3_.tex
\documentclass[review,onefignum,onetabnum]{siamart171218}

\input{ex_shared}

\ifpdf
\hypersetup{
  pdftitle={Multiple solutions and their asymptotics for laminar flows},
  pdfauthor={H. X. Guo, C. F.  Gui, P. Lin and M. F. Zhao}
}
\fi

\begin{document}

\maketitle

\begin{abstract}
The existence and multiplicity of similarity solutions for the steady, incompressible and fully developed laminar flows in a uniformly porous channel with two permeable walls are investigated. We shall focus on the so-called asymmetric case where the upper wall is with an amount of flow injection and the lower wall with a different amount of suction. We show that there exist three solutions designated  as type $I$,  type $II$ and type $III$  for the asymmetric case. The numerical results suggest that a unique solution exists  for the Reynolds number $0\leq R<14.10$  and two  additional solutions appear for  $R>14.10$. The corresponding asymptotic solution for each of the multiple solutions  is  constructed by the method of boundary layer correction or matched asymptotic expansion for the most difficult high Reynolds number case. Asymptotic solutions are all verified by  their corresponding numerical solutions. 
\end{abstract}



\section{Introduction}
Laminar flows in various geometries with porous walls are of fundamental importance to biological organisms, contaminant transports in aquifers and fractures, air circulation in the respiratory system, membrane filtration, control of boundary layer separation, automotive filters, etc. Hence, laminar flows through permeable walls have been extensively studied by researchers  during the past several decades.
The analysis for Navier-Stokes equation which describes the two-dimensional steady laminar flows of a viscous incompressible fluid through a porous channel with uniform injection or suction was initiated by Berman \cite{berman1953laminar}.  He assumed that  the flow is symmetric about the centre line of the channel and  is of similarity form, and reduced the problem to a  fourth order nonlinear ordinary differential equation with four boundary conditions and a cross-flow Reynolds number $R$. He also gave an asymptotic solution for small Reynolds number. Then, numerous studies have been done about the laminar flows in a channel or tube  with permeable walls. Yuan \cite{yuan1956further}, Terrill and Shrestha \cite{terrill1965laminar} and Sellars \cite{sellars1955laminar} obtained an asymptotic solution for the large injection and large suction  cases, respectively. Terrill \cite{terrill1965laminar, terrill1964laminar} and Sherstha \cite{sherstha1967singular}  derived a series of asymptotic solutions using the method of matched asymptotic expansion for the large injection and large suction cases with a transverse magnetic field. 

All these works mentioned above had produced only one solution for each value of $R$.  Raithby \cite{raithby1971laminar}  was the first to find that there is a second solution for  values of $R>12$ in a  numerical investigation of the  flow in a channel  with heat transfer.  Then, Some  studies were also devoted to the analysis of multiple solutions  for the symmetric  porous channel flow problem. Robinson\cite{robinson1976existence}  conjectured that there are three types of solutions  which were classified as type $I$, type $II$ and type  $III$. His conclusion was based on the numerical solutions and he also derived the asymptotic solutions of types $I$ and $II$ for the large suction case. Zarturska et al \cite{zaturska1988flow} and Cox and King \cite{cox1997asymptotic} also analyzed  multiple solutions for the flow in a channel with porous walls.  Lu et al \cite{lu1992asymptotic}  investigated the asymptotic behaviour of the solutions. Lu and Macgillivray \cite{lu1997asymptotic, lu1999uniqueness, lu1999matched, macgillivray1994asymptotic}  mainly obtained the type $III$ solution. Brady and  Acrivos \cite{brady1981steady}  presented three  solutions for the flow in a channel or tube with an accelerating surface velocity. 

As mentioned above, the evidence for multiple solutions was either numerical or asymptotic. Shih \cite{shih1987existence} proved theoretically, applying a fixed point theorem, that there exists only one solution for  injection case for the flow in a channel with porous walls. Topological and shooting methods were used by Hastings et al \cite{hastings1992boundary} to prove the existence of all three of Robinson's conjectured solutions. He also presented the asymptotic behaviour for the flow as $\vert R\vert\rightarrow\infty$. Terrill \cite{terrill1965laminarinjection} proposed a transformation to convert the two-point boundary value problem  into an initial value problem to facilitate the numerical calculation of solutions for an arbitrary  Reynolds number. Based on the transformation proposed by Terrill, Skalak and Wang \cite{skalak1978nonunique} described analytically the number and character of the solutions  for each given Reynolds number under fairly general assumptions for  the symmetrical channel and tube flow. The similar method was used by Cox \cite{cox1991analysis} to analyze the symmetric solutions when the two walls are accelerating equally and when  one wall is accelerating and the other is stationary. The uniqueness of similarity solution was investigated theoretically  by Chenllam  and Liu \cite{chellam2006effect} and their  work mainly considered the symmetric flow in a channel with slip boundary conditions.

All studies mentioned above are for symmetrical flows. The class of asymmetrical flows which may be driven by imposing different  velocities on the walls  turn out to be very interesting. The study of  asymmetric laminar steady flow  may be traced back to Proudman \cite{proudman1960example} who  obtained the asymptotic solutions in the core area.  Then, Terrill and Shrestha \cite{terrill1965laminardiff, terrill1966laminar, shrestha1968laminar} extended Proudman$'$work and constructed  one asymptotic solution using the method of matched asymptotic expansion for the large injection, large suction and mixed cases, respectively. Here the mixed cases means that one wall is with injection while the other is with suction.  Cox \cite{cox1991two} considered the practical case of an impermeable wall opposing a transpiring wall. Waston et al \cite{watson1991laminar}  also investigated the case of asymmetrical flow in a channel which one wall is stationary and the other is accelerating.

The purpose of this paper is not to reconsider any of these previously considered problems, but instead to provide a thorough analysis for the asymmetric flow in a channel of porous walls  with different permeabilities, where the upper wall is with injection and the lower  wall is with suction. We will show that there exist three multiple solutions in this asymmetric case. We also mark them as type $I$, type $II$ and type $III$ solutions as people do for the symmetric case. We should remark here that type $I$, type $II$ and type $III$ solutions for the asymmetric case are much different from those for the symmetric case. We will numerically give the range of the Reynolds number where there exist three solutions. We will then construct asymptotic solutions for each solution for the most difficult case of high Reynolds number and numerically validate the constructed solutions. The paper is organized as follows.  In \cref{sec:model},   a similarity transformation is introduced and  the Navier -Stokes equation is reduced to a single fourth order nonlinear ordinary differential equation with a Reynolds number $R$ and four boundary conditions. In \cref{sec:existence}, we  theoretically analyze that there exist three solutions of similarity transformed equation under fairly general assumptions. In \cref{num sol}, we compute the multiple solutions numerically. We also  sketches  velocity profiles and streamlines for these asymmetric flows. In \cref{asy sol}, for the most difficult high Reynolds number case,  the asymptotic solution for each type  of multiple  solutions will be constructed using the method of boundary layer correction or matched asymptotic expansion.  In \cref{compa},  all the asymptotic solutions are verified by numeral solutions and meanwhile these asymptotic solutions may serve as a validation for the numerical method used in the paper. 

\section{Mathematical formulation}
\label{sec:model}

We consider the two-dimensional, viscous, incompressible asymmetric  laminar flows  in a porous and  elongated rectangular channel. The channel exhibits  a sufficiently small depth-width ratio of semi-height $h$  to  length $L$.  Despite the channel's finite body length, it is reasonable to assume a  semi-infinite length in order to neglect the influence of the opening at the end \cite{uchida1977unsteady}.  The flow is driven by uniform injection through the upper wall of the channel  with speed $-v_{2}$ and suction through the lower  wall with speed $-v_{1}$,  where $v_2>v_1>0$. We define an asymmetric parameter  $a=\frac{v_{1}}{v_{2}}$, where $0<a<1$. With $\tilde{x}$ representing the streamwise direction and $\tilde{y}$ the transverse direction, the corresponding streamwise and transverse velocity components are defined as $u$ and $v$, respectively.  The streamwise velocity is zero at the closed headend $(\tilde{x}=0)$.  numerical

The equations of the continuity and momentum  for the steady laminar flows of an incompressible viscous fluid through a porous channel  are \cite{terrill1966laminar}
\begin{align}
\nabla\cdot \textit{\textbf{V}}&=0,   \label{continuity}\\
(\textit{\textbf{V}} \cdot \nabla)\textit{\textbf{V}} &= -\frac{1}{\rho} \textrm{ } \nabla \textrm{ } p+ \nu \nabla^2 \textit{\textbf{V}},  \label{momentum}
\end{align}
where the symbol $ \textit{\textbf{V}}=(u, v)$   represents the velocity vector,  $p$ the pressure, $\rho$ the density and $\nu$ the viscosity of the fluid.
The boundary conditions necessary for describing the asymmetric flow and solving the continuity and momentum equations are  
\begin{equation}
u(\tilde{x}, -h) =0,  \  v(\tilde{x}, -h )=- v_1,  \ u(\tilde{x}, h ) =0,   \   v(\tilde{x}, h ) =- v_2.    \label{pdebc}
\end{equation}

As we know, the  study of the fluid flow equations with a high Reynolds number is most challenging.  A similarity transformation provides a way to explicitly express the solutions for the channel flow. Such  explicit  solutions are often preferred in understanding the fluid properties especially the boundary layers. For this purpose we introduce a streamfunction and express the velocity components in terms of the streamfunction \cite{majdalani2003moderate}:
\begin{equation}
\phi =\frac{\nu \tilde{x}}{h}F(y), \nonumber
\end{equation}
where $y=\frac{\tilde{y}}{h}$ which is the non-dimensional transverse coordinate and $F(y)$ is independent to the streamwise coordinate.
Then the velocity components are given by 
\begin{equation}
u=\frac{\partial \phi}{\partial \tilde{y}}=\frac{\nu \tilde{x}}{h^2}F'(y), \quad v=-\frac{\partial \phi}{\partial \tilde{x}}=-\frac{\nu}{h}F(y), \label{velocity}
\end{equation}
so that the continuity equation (\ref{continuity}) is  satisfied.
Substituting (\ref{velocity}) into equations  (\ref{continuity}) and (\ref{momentum}) and eliminating the pressure term,  a third order nonlinear ordinary differential equation with a parameter and an integration constant is developed.  It is 
\begin{equation}
f'''+ R(ff''-f'^2)=K, \label{ode1}
\end{equation}
where  $K$ is an integration constant, $R=\frac{hv_2}{\nu}$ is the Reynolds number based on the fluid velocity $v_2$  through the upper wall and the semi-height $h$ of channel and $f=\frac{F}{R}$.
The boundary conditions given by (\ref{pdebc})  can now be updated to the normalized  form 
\begin{equation}
f(-1)=a, \  f'(-1)=0, \  f(1)=1, \  f'(1)=0. \label{bc}
\end{equation}

\section{Existence of multiple solutions}
\label{sec:existence}

Skalak \cite{skalak1978nonunique}, Cox \cite{cox1991analysis} and Chenllam \cite{chellam2006effect} considered symmetric flow in a channel with porous walls, accelerating walls  and slip boundary conditions, respectively.  
In this section, we extend previous analysis \cite{skalak1978nonunique, cox1991analysis, chellam2006effect} to investigate asymmetric  flow in the channel with porous walls  and to discuss  the existence  of multiple solutions. 
 
The two-point boundary value problem,  (\ref{ode1}) and (\ref{bc}),  can be converted into an initial value problem.
Rescalling (\ref{ode1}) and (\ref{bc}) by introducing $f(y)=\frac{1}{2}bg(\xi)/R$ and $\xi=\frac{1}{2}b(y+1)$\cite{terrill1965laminarinjection} \cite{skalak1978nonunique}:
\begin{equation}
g'''+ gg''-g'^2=k, \label{eqn:3-1}
\end{equation}
\begin{equation}
g(0)=\frac{2aR}{b}, \ g'(0)=0, \ g(b)=\frac{2R}{b}, \  g'(b)=0,
\end{equation}
where $k=\frac{16RK}{b^4}$, $a>0$,  $b>0$ and $R>0$.
Assume all the initial conditions of (\ref{eqn:3-1}) can be
\begin{equation}
g(0)=2 a R/b,\quad g'(0)=0, \quad g''(0)=A, \quad g'''(0)=B,  \label{eqn:3-3}
\end{equation}
where $A,B\neq0$. 
If we find a  point $\xi^*$ such that $g'(\xi)=0$, we obtain $b$ by setting $\xi^*=b$ and the value of $g(b)$ at this point gives the Reynolds  number $R=\frac{1}{2}bg(b)$. Then, we can obtain the original solution $f(y)$. Since our analysis is restricted to positive $R$,  there must be $g(b)>0$ (i.e. $g(\xi^*)>0$). Thus,  we will discuss analytically the number of possible roots $\xi^*$ of $g'(\xi)$ (each with $g(\xi^*)>0$) covering the entire range of $\xi>0$.

Let $h_1(\xi)=g'(\xi)$, $h_2(\xi)=g''(\xi)$, $h_3(\xi)=h'''(\xi)$, $h_4(\xi)=g^{(4)}(\xi)$ and $h_5(\xi)=g^{(5)}(\xi)$ for all $\xi\geq0$. 
Differenting equation (\ref{eqn:3-1}) twice gives
\begin{align}
&h_4=h_1h_2-gh_3, \label{eqn:3-4}\\
&h_5+gh_4=(h_2)^2.  \label{eqn:3-5}
\end{align}

\begin{lemma}\label{lem:3-1}
       If $B<0$, then $g^{(4)}(\xi)>0$ for all $\xi\geq0$. In particular,  $g'''(\xi)$ is a strictly increasing function of $\xi$ on $[0,+\infty)$.	
\end{lemma}

\begin{proof}
	By equation \eqref{eqn:3-4}, since $g'(0)=0$, then $h_4(0)=g'(0)g''(0)-g(0)h_3(0)=-g(0)B$. Since $g(0)>0$ and $B<0$, then $h_4(0)>0$.  By equation \eqref{eqn:3-5}, then $h_4'+gh_4\geq0$ for all $\xi\geq0$, that is, $\displaystyle \left(e^{\int_0^\xi g(t)dt}h_4(\xi)\right)'\geq0$ for all $\xi\geq0$, which implies that $e^{\int_0^\xi g(t)dt}h_4(\xi)\geq h_4(0)>0$ for all $\xi\geq0$. Hence $h_4(\xi)=g^{(4)}(\xi)>0$ for all $\xi\geq0$.	
\end{proof}

\begin{remark}
The proof of \cref{lem:3-1} only uses the condition $g(0)B<0$.	
\end{remark}

\begin{proposition}\label{prop:3-1}
Assume that $A>0$ and $B<0$.
	\begin{itemize}	
\item[(a)] If there exists some $x_0>0$ such that $g'''(x_0)=0$, then there is no point $\zeta>x_{0}$ such that $g'(\zeta)=0$ and $g(\zeta)>0$.
\item[(b)]  If  $g'''(\xi)<0$ for all $\xi\geq0$ and $h_2(\alpha)=0$ for some $\alpha>0$, then there exists some $\zeta>\alpha$ such that $g'(\zeta)=0$ and $g(\zeta)>0$.
       \end{itemize}	
\end{proposition}

\begin{proof}
(a) By \cref{lem:3-1}, then $h_4(\xi)>0$ for all $\xi\geq0$. Since $h_3(x_0)=0$, then  $h_3(\xi)<0$ for all $\xi\in[0,x_0)$ and $h_3(\xi)>0$ for all $\xi>x_0$, which implies that $h_2(\xi)$ is strictly decreasing on $[0,x_0)$ and strictly increasing on $(x_0,+\infty)$. Hence $x_0$ is the unique global minimum point of $h_2(\xi)$ on $[0,+\infty)$. Since $h_3(x_0)=0$, by equation \eqref{eqn:3-4} and \cref{lem:3-1}, then $h_4(x_0)=h_1(x_0)h_2(x_0)-g(x_0)h_3(x_0)=h_1(x_0)h_2(x_0)>0$. \\	
Let's assume that there is some $\zeta>x_0$ such that $h_1(\zeta)=g'(\zeta)=0$ and $g(\zeta)>0$.  Since $g'(0)=g'(\zeta)=0$, by Lagrange's mean value theorem, then there exists some $\alpha\in(0,\zeta)$ such that $h_2(\alpha)=0$. Since $x_0$ is the unique global minimum point of $h_2(\xi)$ on $[0,+\infty)$ and $h_2(x_0)\neq0$, then $h_2(x_0)<0$. Since $h_4(\xi)>0$ for all $\xi\geq0$ and $h_3(\xi)>0$ for all $\xi>x_0$, then $\displaystyle \lim_{\xi\rightarrow+\infty} h_2(\xi)=+\infty$. Since $h_2(0)=A>0$, $h_2(x_0)<0$, and $h_4(\xi)>0$ for all $\xi\geq0$, then there exists a unique $\beta>x_0$ such that $h_2(\xi)<0$ for all $\alpha<\xi<\beta$, and $h_2(\xi)>0$ for all $0\leq \xi<\alpha$ and all $\xi>\beta$. 
		Since $h_1(x_0)h_2(x_0)>0$ and $h_2(x_0)<0$, then $h_1(x_0)<0$. Since $h_1(\zeta)=0$, by equation \eqref{eqn:3-4}, then $h_4(\zeta)=h_1(\zeta)h_2(\zeta)-g(\zeta)h_3(\zeta)=-g(\zeta)h_3(\zeta)$. Since $h_4(\xi)>0$ for all $\xi>0$ and $g(\zeta)>0$, then $h_3(\zeta)<0$, which implies that $\zeta<x_0$, this leads to a contradiction.

(b) Since $g'''(\xi)<0$ for all $\xi\geq0$ and $h_2(\alpha)=0$ for some $\alpha>0$, then $h_2(\xi)>0$ for all $0\leq \xi<\alpha$ and $h_2(\xi)<0$ for all $\xi>\alpha$, which implies that $h_1(\xi)$ is strictly increasing on $[0,\alpha)$ and strictly decreasing on $(\alpha, +\infty)$. Hence $\alpha$ is the unique global maximum point of $h_1(\xi)$ on $[0,+\infty)$. Since $h_1(0)=0$, then $h_1(\alpha)>0$. Since $h_2(\xi)<0$ for all $\xi>\alpha$ and $h_3(\xi)<0$ for all $\xi\geq0$, then $\displaystyle \lim_{\xi\rightarrow+\infty} h_1(\xi)=-\infty$. Since $h_1(\alpha)>0$ and $h_2(\xi)<0$ for all $\xi>\alpha$, then there exists q unique $\zeta>\alpha$ such that $g'(\zeta)=h_1(\zeta)=0$. Since $h_1(\zeta)=0$, by equation \eqref{eqn:3-4}, then $h_4(\zeta)=h_1(\zeta)h_2(\zeta)-g(\zeta)h_3(\zeta)=-g(\zeta)h_3(\zeta)$. Since $\zeta>\alpha$, $h_3(\xi)<0$ for all $\xi\geq0$ and \cref{lem:3-1}, then $g(\zeta)>0$. Therefore, in this case, there exists a solution of (\ref{ode1}) and (\ref{bc}).  We will designate this solution as type $I$ solution.
\end{proof}

\begin{proposition}\label{prop:3-2}
	Assume that $A<0$ and $B<0$.
	\begin{itemize}
		\item[(a)] Then there exists some $x_0>0$ such that $g'''(x_0)=0$. 
		  
		\item[(b)] Then there is no point $\zeta>0$ such that $g'(\zeta)=0$ and $g(\zeta)>0$.	
	\end{itemize}	
\end{proposition}

\begin{proof}
(a) If the statement is not right, since $B=g'''(0)<0$, then $h_3(\xi)<0$ for all $\xi\geq0$. Since $A=g''(0)<0$, then $h_2(\xi)\leq A<0$ for all $\xi\geq0$, which implies that $h_1(\xi)-h_1(0)\leq A \xi$ for  all $\xi\geq0$. Since $h_1(0)=0$, then $h_1(\xi)<0$ and $h_1(\xi)\leq A \xi$ for all $\xi\geq0$, which implies that $g(\xi)-g(0)\leq \frac{A}{2}\xi^2$ for all $\xi\geq0$. Hence $\displaystyle \lim_{\xi\rightarrow+\infty} g(\xi)=-\infty$. By equation \eqref{eqn:3-5}, then $h_5=(h_2)^2-gh_4$ for all $\xi\geq 0$. By \cref{lem:3-1}, then $h_4(\xi)>0$ for all $\xi\geq0$. Since $\displaystyle \lim_{\xi\rightarrow+\infty} g(\xi)=-\infty$, then there exists some $R>0$ such that $h_5(\xi)>0$ for all $\xi\geq R$. Since $h_4(\xi)>0$ and $h_5(\xi)>0$ for all $\xi\geq R$, then $\displaystyle \lim_{\xi\rightarrow+\infty} h_3(\xi)=+\infty$, this leads to a contradiction.
		
               (b) By the result of part (a) then there exists some $x_0>0$ such that $g'''(x_0)=0$. Since $h_4(\xi)\geq0$ for all $\xi\geq0$, then $h_3(\xi)<0$ for all $\xi\in[0,x_0)$ and $h_3(\xi)>0$ for all $\xi>x_0$, which implies that $h_2(\xi)$ is strictly decreasing on $[0,x_0)$ and strictly increasing on $(x_0,+\infty)$. Hence $x_0$ is the unique global minimum point of $h_2(\xi)$ on $[0,+\infty)$. Since $g''(0)=A<0$ and $h_3(\xi)<0$ for all $\xi\in[0,x_0)$, then $h_2(\xi)<A<0$ for all $\xi\in(0,x_0]$. Since $h_1(0)=0$, then $h_1(\xi)\leq A \xi<0$ for all $\xi\in(0,x_0]$. Since $h_3(\xi)>0$ for all $\xi>x_0$ and $h_4(\xi)>0$ for all $\xi\geq0$, then $\displaystyle \lim_{\xi\rightarrow+\infty} h_2(\xi)=+\infty$. Since $h_2(x_0)<0$ and $h_3(\xi)>0$ for all $\xi>x_0$, then there exists a unique $x_1>x_0$ such that $h_2(\xi)<0$ for all $\xi\in[x_0,x_1)$, $h_2(x_1)=0$, and $h_2(\xi)>0$ for all $\xi>x_1$. Since $h_2(\xi)<A<0$ for all $\xi\in(0,x_0]$, then $h_2(\xi)<0$ for all $0\leq \xi<x_1$ and $h_2(\xi)>0$ for all $\xi>x_1$, which implies that $h_1(\xi)$ is strictly decreasing on $[0,x_1]$ and strictly increasing on $[x_1,+\infty)$. Hence $x_1$ is the unique global minimum point of $h_1(\xi)$ on $[0,+\infty)$. Since $h_1(\xi)\leq A\xi<0$ for all $\xi\in(0,x_0]$, then $h_1(x_1)<0$. Since $h_2(\xi)>0$ for all $\xi>x_1$ and $h_3(\xi)>0$ for all $\xi>x_0$, then $\displaystyle \lim_{\xi\rightarrow+\infty} h_1(\xi)=+\infty$. Since $h_1(\xi)<0$ for all $0<\xi\leq x_1$ and $h_2(\xi)>0$ for all $\xi>x_1$, then there exists a unique $x_2>x_1$ such that $h_1(x_2)=0$. So we know that $x_2$ is the unique solution of $h_1(\xi)=0$ for all $\xi>0$. On the other hand, since $x_2>x_1>x_0$, then $h_4(x_2)>0$ and $h_3(x_2)>0$. Since $h_1(x_2)=0$, by equation \eqref{eqn:3-4}, then $h_4(x_2)=h_1(x_2)h_2(x_2)-g(x_2)h_3(x_2)=-g(x_2)h_3(x_2)>0$. Since $h_3(x_2)>0$, then $g(x_2)<0$. Since $x_2$ is the unique solution of $h_1(\xi)=0$ for all $\xi>0$, and  $g(x_2)<0$, then there is no point $\zeta>0$ such that $g'(\zeta)=0$ and $g(\zeta)>0$.			
\end{proof}

\begin{proposition}\label{prop:3-3}
	Assume that $B>0$.
	\begin{itemize}
		\item[(a)] If $h_4(\xi)\neq0$ for all $\xi>0$, then $h_4(\xi)<0$ for all $\xi\leq0$.

		\item[(b)] If $h_4(\xi_{0})=0$ for some $\xi_{0}>0$, then $h_4(\xi)<0$ for all $0\leq \xi<\xi_{0}$ and $h_4(\xi)>0$ for all $\xi>\xi_{0}$.	
	\end{itemize}	
\end{proposition}

\begin{proof}
		(a) By equation \eqref{eqn:3-4}, since $g'(0)=0$, then $h_4(0)= h_1(0)h_2(0)-g(0)h_3(0)$ $=-g(0)B$. Since $g(0)>0$ and $B>0$, then $h_4(0)<0$. By the assumption that $h_4(\xi)\neq0$ for all $\xi>0$, then $h_4(\xi)<0$ for all $\xi\leq0$.		
		(b) By the proof of part (a), we know that $h_4(0)<0$. Let $\Sigma=\{\xi>0: h_4(\xi)=0 \}$, by the assumption, then $\xi_0\in \Sigma$. Since $h_4(0)<0$, by the continuity, we know that $\displaystyle \eta:=\inf_{\xi\in\Sigma}\ \xi\in (0,\xi_0]$, which implies that $\xi_0\geq\eta$, $h_4(\eta)=0$ and $h_4(\xi)<0$ for all $0\leq \xi<\eta$. By equation \eqref{eqn:3-5}, then $h_4'+gh_4\geq0$ for all $\xi\geq0$, that is, $\displaystyle \left(e^{\int_\eta^\xi g(t)dt}h_4(\xi)\right)'\geq0$ for all $\xi\geq0$, which implies that $e^{\int_\eta^\xi g(t)dt}h_4(\xi)\geq h_4(\eta)=0$ for all $\xi\geq\eta$. Hence $h_4(\xi)=g^{(4)}(\xi)\geq0$ for all $\xi\geq \eta$. 

\begin{claim}\label{clm:1}		
 $h_4(\xi)>0$ for all $\xi>\eta$.
 \end{claim}
	\begin{proof}	
	If not, since $h_4(\xi)\geq0$ for all $\xi\geq \eta$, then there exists some $\alpha>\eta$ such that $h_4(\alpha)=0$. If $h_4(\xi)\equiv 0$ for all $\eta\leq \xi\leq \alpha$, by equation \eqref{eqn:3-5}, then $h_2(\xi)\equiv 0$ for all $\eta\leq \xi\leq\alpha$. Rewrite equation \eqref{eqn:3-4} as $h_2''+gh_2'-h_1h_2=0$ for all $\xi\geq0$, since $h_2(\eta)=h_2'(\eta)=0$, by the uniqueness theorem of solutions to the ODE, then $h_2(\xi)=0$ for all $\xi\geq0$, which implies that $B=h_3(0)=h_2'(0)=0$, this leads to a contradiction. Since $h_4(\xi)\geq0$ for all $\xi\geq \eta$, then there exists some $\beta\in(\eta,\alpha)$ such that $h_4(\beta)>0$. Since  $\displaystyle \left(e^{\int_\eta^\xi g(t)dt}h_4(\xi)\right)'\geq0$ for all $\xi\geq0$ and $h_4(\beta)>0$, then $h_4(\xi)>0$ for all $\xi>\beta$. Since $\alpha>\beta$, then $h_4(\beta)<h_4(\alpha)=0$, this leads to a contradiction. Therefore, we know that $h_4(\xi)>0$ for all $\xi>\eta$.
       \end{proof}		
		
		 Since $h_4(\xi_0)=h_4(\eta)=0$, $\xi_0\geq\eta$, by \cref{clm:1}, then $\xi_0=\eta$. By the definition of $\eta$ and \cref{clm:1}, then $h_4(\xi)<0$ for all $0\leq \xi<\xi_0$ and $h_4(\xi)>0$ for all $\xi>\xi_0$.
\end{proof}

\begin{proposition}\label{prop:3-4}
	Assume that $A>0$ and $B>0$, if $g^{(4)}(\xi)<0$ for all $\xi\geq0$, then $g'''(\xi)>0$ for all $\xi\geq0$. In particular, since $g''(0)=A>0$, then $g''(\xi)>A>0$ for all $\xi>0$. Moreover, since $g'(0)=0$, then $g'(\xi)>0$ for all $\xi>0$.
\end{proposition}

\begin{proof}
		If not, since $h_3(0)=B>0$ and $h_4(\xi)<0$ for all $\xi\geq0$, then there exists a unique $x_0 >0$ such that $h_3(\xi)>0$ for all $0\leq \xi<x_0$, $h_3(x_0)=0$, and $h_3(\xi)<0$ for all $\xi>x_0$. Since $h_2(0)=A>0$ and $h_3(\xi)>0$ for all $0\leq \xi< x_0$ , then $h_2(\xi)>0$ for all $0\leq \xi\leq x_0$. Since $h_1(0)=0$, then $h_1(\xi)>0$ for all $0\leq \xi\leq x_0$, which implies that $h_1(x_0)h_2(x_0)>0$. Since $h_3(x_0)=0$ and $h_4(\xi)<0$ for all $\xi\geq0$, by equation \eqref{eqn:3-4}, we know that $h_1(x_0)h_2(x_0)=h_4(x_0)-g(x_0)h_3(x_0)=h_4(x_0)<0$, this leads to a contradiction.
\end{proof}

\begin{proposition}\label{prop:3-5}
	Assume that $A>0$ and $B>0$, and there exists some $\xi_{0}>0$ such that $g^{(4)}(\xi)<0$ for all $\xi\in(0,\xi_{0})$ and $g^{(4)}(\xi)>0$ for all $\xi>\xi_{0}$. Then $\xi_{0}$ is the unique global minimum point of $g'''(\xi)$ on $[0,\infty)$ and $g'''(\xi_{0})>0$.  In particular, since $g''(0)=A>0$, then $g''(\xi)>A>0$ for all $\xi>0$. Moreover, since $g'(0)=0$, then $g'(\xi)>0$ for all $\xi>0$.	
\end{proposition}

\begin{proof}
	Since  $g^{(4)}(\xi)<0$ for all $\xi\in(0,\xi_0)$ and $g^{(4)}(\xi)>0$ for all $\xi>\xi_0$, then $g'''(\xi)$ is strictly decreasing on $[0,\xi_0]$ and strictly increasing on $[\xi_0,+\infty)$, which implies that $\xi_0$ is the unique global minimum point of $g'''(\xi)$ on $[0,\infty)$. Now let's decide the sign of $h_3(\xi_0)=g'''(\xi_0)$. If $h_3(\xi_0)\leq0$, since $h_3(0)=B>0$ and $g^{(4)}(\xi)<0$ for all $\xi\in[0,\xi_0)$, then there exists a unique $\eta\in(0,\xi_0]$ such that $h_3(\xi)>0$ for all $0\leq \xi<\eta$, $h_3(\eta)=0$, and $h_3(\xi)<0$ for all $\xi_0>\xi>\eta$. Since $h_2(0)=A>0$, then $h_2(\xi)>A>0$ for all $0\leq \xi\leq \eta$. Since $h_1(0)=0$, then $h_1(\xi)>0$ for all $0\leq \xi\leq \eta$, which implies that $h_1(\eta)h_2(\eta)>0$. Since $h_3(\eta)=0$, $h_4(\xi)<0$ for all $0\leq \xi<\xi_0$ and $0\leq\eta\leq\xi_0$, by equation \eqref{eqn:3-4}, we know that $h_1(\eta)h_2(\eta)=h_4(\eta)-g(\eta)h_3(\eta)=h_4(\eta)\leq 0$, this leads to a contradiction.
\end{proof}

\begin{proposition}\label{prop:3-6}
	Assume that $A<0$ and $B>0$, and $g^{(4)}(\xi)<0$ for all  $\xi\geq0$.
	\begin{itemize}
		\item[(a)] If $g'''(x_0)=0$ for some $x_0>0$, then there is no point $\zeta>0$ such that $g'(\zeta)=0$ and $g(\zeta)>0$. In particular, there is no point $\zeta>x_0$ such that $g'(\zeta)=0$ and $g(\zeta)>0$.
		
		\item[(b)] If $g'''(\xi)>0$ for all $\xi\geq0$, then there is no point $\zeta>0$ such that $g'(\zeta)=0$ and $g(\zeta)>0$.	
	\end{itemize}	
\end{proposition}

\begin{proof} 
    	     (a) If there exists some $\zeta>0$ such that $g'(\zeta)=0$ and $g(\zeta)>0$, since $h_4(\xi)<0$ for all $\xi\geq0$, by equation \eqref{eqn:3-4}, then $h_4(\zeta)=h_1(\zeta)h_2(\zeta)-g(\zeta)h_3(\zeta)=-g(\zeta)h_3(\zeta)<0$. Since $g(\zeta)>0$, then $h_3(\zeta)>0$. Since $h_4(\xi)<0$ for all $\xi\geq0$ and $h_3(x_0)=0$, then $0<\zeta<x_0$.
    	Since $g'(0)=g'(\zeta)=0$, by Lagrange's mean value theorem, then $h_2(\eta)=0$ for some $\eta\in(0,\zeta)$. Since $h_4(\xi)<0$ for all $\xi\geq0$ and $h_3(x_0)=0$, then $h_3(\xi)>0$ for all $\xi\in[0, x_0)$ and $h_3(\xi)<0$ for all $\xi>x_0$, which implies that $h_2(\xi)$ is strictly increasing on $[0,x_0]$ and strictly decreasing on $[x_0,+\infty)$. Hence $x_0$ is the unique global maximum point of $h_2(\xi)$ on $[0,+\infty)$. Since $h_2(\eta)=0$, then $h_2(x_0)>0$. Since $h_4(\xi)<0$ for all $\xi\geq0$ and $h_3(x_0)=0$, by equation \eqref{eqn:3-4}, then $h_4(x_0)=h_1(x_0)h_2(x_0)-g(x_0)h_3(x_0)=h_1(x_0)h_2(x_0)<0$. Since $h_2(x_0)>0$, then $h_1(x_0)<0$. Since $h_3(\xi)>0$ for all $\xi\in[0,x_0)$, then $h_2(\xi)<0$ for all $\xi\in[0,\eta)$ and $h_2(\xi)>0$ for all $\xi\in(\eta,x_0)$, which implies that $h_1(\xi)$ is strictly decreasing on $[0,\eta)$ and strictly increasing on $(\eta,x_0)$. Since $h_1(0)=0$, then $h_1(\xi)<0$ for all $\xi\in(0,x_0)$. Since $\zeta\in(0,x_0)$, then $g'(\zeta)=h_1(\zeta)<0$, this leads to a contradiction.
    	
    	   (b) If there exists some $\zeta>0$ such that $g'(\zeta)=0$ and $g(\zeta)>0$, since $g'(0)=0$, by Lagrange's mean value theorem, then $h_2(\eta)=0$  for  some  $\eta\in(0,\zeta)$. Since $h_3(\xi)>0$ for all $\xi\geq0$, then $h_2(\xi)<0$ for all $\xi\in[0,\eta)$ and $h_2(\xi)>0$ for all $\xi>\eta$, which implies that $h_1(\xi)$ is strictly decreasing on $[0,\eta]$ and strictly increasing on $[\eta,+\infty)$. Since $\eta<\zeta$ and $h_1(0)=h_1(\zeta)=0$, then $h_1(\xi)<0$ for all $0<\xi<\zeta$ and $h_1(\xi)>0$ for all $\xi>\zeta$, which implies that $\zeta$ is the unique global minimum point of $g(\xi)$ on $[0,+\infty)$. Since $\Delta:=g(\zeta)>0$, then $g(\xi)\geq\Delta>0$ for all $\xi\geq0$. Since $h_4(\xi)<0$ for all $\xi\geq0$, by equation  \eqref{eqn:3-5}, then $h_5(\xi)=(h_2(\xi))^2-g(\xi)h_4(\xi)>0$ for all $\xi\geq0$.  Since $h_3(\xi)>0$ and $h_4(\xi)<0$ for all $\xi\geq0$, then there exists some $L\geq0$ such that $\displaystyle \lim_{\xi\rightarrow+\infty} h_3(\xi)=L$ and $h_3(\xi)>L\geq0$ for all $\xi\geq0$. By Lagrange's mean value theorem, there exists some sequence $\{y_n\}_{n=1}^\infty$ such that $\displaystyle \lim_{n\rightarrow\infty} y_n=+\infty$ and $\displaystyle \lim_{n\rightarrow\infty} h_4(y_n)=0$. Since $h_5(\xi)>0$ for all $\xi\geq0$, then $h_4(\xi)$ is strictly increasing on $[0,+\infty)$, which implies that $\displaystyle \lim_{\xi\rightarrow+\infty} h_4(\xi)=0$. By Lagrange's mean value theorem, there exists some sequence $\{x_n\}_{n=1}^\infty$ such that $\displaystyle \lim_{n\rightarrow\infty} x_n=+\infty$ and $\displaystyle \lim_{n\rightarrow\infty} h_5(x_n)=0$. Since $g(\xi)>0$ and $h_4(\xi)<0$ for all $\xi\geq0$, then $\displaystyle \limsup_{n\rightarrow\infty}[h_5(x_n)+g(x_n)h_4(x_n)]\leq0$. On the other hand, since $h_2(\xi)>0$ for all $\xi>\eta$ and $h_3(\xi)>0$ for all $\xi\geq0$, then $\displaystyle \limsup_{n\rightarrow\infty}\ (h_2(x_n))^2>0$, which contradicts with equation \eqref{eqn:3-5}.  
 \end{proof}

\begin{proposition}\label{prop:3-7}
	Assume that $A<0$ and $B>0$, and there exists some $\xi_{0}>0$ such that $g^{(4)}(\xi)<0$ for all $\xi\in(0,\xi_{0})$ and $g^{(4)}(\xi)>0$ for all $\xi>\xi_{0}$.
	\begin{itemize}
		\item[(a)] If $g'''(x_0)\geq0$ for some $x_0\geq \xi_0$, then there is no point  {$\zeta>x_0$} such that $g'(\zeta)=0$ and $g(\zeta)>0$. In particular, if $g'''(\xi_0)\geq 0$, then there is no point  {$\zeta>\xi_0$} such that $g'(\zeta)=0$ and $g(\zeta)>0$.
		
	      \item[(b)]{ If $g'''(\xi)<0$ for all $\xi\in[\xi_0, \infty)$  and $g''(\xi)$ has only one zero on $[0, \infty)$, then there is  no point $\zeta>\xi_0$ such that $g'(\zeta)=0$ and $g(\zeta)>0$.  }
	    
	    \item[(c)]  If $g'''(\xi)<0$ for all $\xi\in[\xi_0, \infty)$ and $g''(\alpha)=0$ and $g'(\alpha)>0$ for some point $\alpha>\xi_0$, then there exists a unique $\zeta>\alpha$ such that $g'(\zeta)=0$ and $g(\zeta)>0$.  In particular, there exists  a unique $\gamma\in(0,\zeta)$ such that $g'(\gamma)=0$. If $\gamma<\xi_0$, then  $g(\gamma)<0$. If $\gamma>\xi_0$, then $g(\gamma)>0$.	    
\end{itemize}	
\end{proposition}

\begin{proof}
	(a) Assume that there is some point $\zeta>x_0$ such that $g'(\zeta)=0$ and $g(\zeta)>0$, since $h_4(\xi)=g^{(4)}(\xi)>0$ for all $\xi>\xi_0$ and $g'''(x_0)\geq0$, then $h_3(\zeta)>0$. Since $h_4(\xi)>0$ for all $\xi>\xi_0$, by equation \eqref{eqn:3-4}, then $h_4(\zeta)=h_1(\zeta)h_2(\zeta)-g(\zeta)h_3(\zeta)=-g(\zeta)h_3(\zeta)>0$. Since $g(\zeta)>0$, then $h_3(\zeta)<0$, this leads to a contradiction.
			
	(b) Assume that there is some point $\zeta>\xi_0$ such that $g'(\zeta)=0$ and $g(\zeta)>0$, since $g'(0)=0$, by Lagrange's mean value theorem, then there exists some $\alpha\in(0,\zeta)$ such that $h_2(\alpha)=g''(\alpha)=0$. Since $g''(0)=A<0$ and $g''(\xi)$ has only one zero on $[0, \infty)$, then $h_2(\xi)<0$ for all $\xi\neq\alpha$, which implies that $\alpha$ is the unique global maximum point of $g''(\xi)$. Hence $h_3(\alpha)=0$ and $h_4(\alpha)\leq0$. Since $g^{(4)}(\xi)<0$ for all $\xi\in(0,\xi_0)$ and $g^{(4)}(\xi)>0$ for all $\xi>\xi_0$, then $0<\alpha\leq \xi_0$. Since $h_3(\alpha)=0$ and $h_3(\xi)=g'''(\xi)<0$ for all $\xi\geq\xi_0$, then $0<\alpha<\xi_0$, which implies that $h_4(\alpha)<0$. On the other hand, since $h_2(\alpha)=h_3(\alpha)=0$, by equation \eqref{eqn:3-4}, then $h_4(\alpha)=h_1(\alpha)h_2(\alpha)-g(\alpha)h_3(\alpha)=0$, this leads to a contradiction.
	 
	 (c) Since $g^{(4)}(\xi)<0$ for all $\xi\in(0,\xi_0)$ and $g^{(4)}(\xi)>0$ for all $\xi>\xi_0$, then $h_3(\xi)$ is strictly decreasing on $[0,\xi_0]$ and strictly increasing on $[\xi_0,+\infty)$, which implies that $\xi_0$ is the unique global minimum point of $h_3(\xi)$ on $[0,+\infty)$. Since $h_3(\xi_0)<0$, $h_3(0)=B>0$ and $h_4(\xi)<0$ for all $0\leq \xi<\xi_0$, then there exists a unique $x_1\in(0,\xi_0)$ such that $h_3(\xi)>0$ for all $0\leq \xi<x_1$ and $h_3(\xi)<0$ for all $x_1<\xi<\xi_0$. Since $h_3(\xi)<0$ for all $\xi\geq\xi_0$, then $h_3(\xi)>0$ for all $0\leq \xi<x_1$ and $h_3(\xi)<0$ for all $\xi>x_1$, which implies that $h_2(\xi)$ is strictly increasing on $[0,x_1]$ and strictly decreasing on $[x_1,+\infty)$. Since $h_2(0)=A<0$ and $h_2(\alpha)=0$, then $h_2(\xi)$ has either $1$ or $2$ zeros on $[0,+\infty)$. If $h_2(\xi)$ has only one zero on $[0,+\infty)$, since $h_2(0)=A<0$ and $h_2(\alpha)=0$, then $\alpha$ is the only global maximum point of $h_2(\xi)$ on $[0,+\infty)$. Since $h_2(\xi)$ is strictly increasing on $[0,x_1]$ and strictly decreasing on $[x_1,+\infty)$, then $\alpha=x_1$, which implies that $\alpha=x_1<\xi_0$, this contradicts $\alpha>\xi_0$. If $h_2(\xi)$ has two zeros on $[0,+\infty)$, then there exists a unique $\beta\in(0,x_1)$ such that $h_2(\xi)>0$ for all $\beta<\xi<\alpha$ and $h_2(\xi)<0$ for all $\xi\in(0,\beta)\cup(\alpha,+\infty)$.  Since $h_2(\xi)<0$ for all $\xi>\alpha$ and $h_3(\xi)<0$ for all $\xi>x_1$, then $\displaystyle \lim_{\xi\rightarrow+\infty} h_1(\xi)=-\infty$. Since $h_1(\alpha)=g'(\alpha)>0$ and $h_2(\xi)<0$ for all $\xi>\alpha$, then there exists a unique $\zeta>\alpha$ such that $h_1(\zeta)=g'(\zeta)=0$. Since $\zeta>\alpha>\xi_0$ and $h_4(\xi)>0$ for all $\xi>\xi_0$, by equation \eqref{eqn:3-4}, then $h_4(\zeta)=h_1(\zeta)h_2(\zeta)-g(\zeta)h_3(\zeta)=-g(\zeta)h_3(\zeta)>0$. Since $\zeta>\alpha>\xi_0$ and $h_3(\xi)<0$ for all $\xi\geq\xi_0$, then $g(\zeta)>0$.
	 
	   Since $h_1(0)=0$ and $h_2(\xi)<0$ for all $0<\xi<\beta$, then $h_1(\xi)<0$ for all $0<\xi\leq\beta$. Since $h_1(\alpha)>0$ and $h_2(\xi)$ is strictly increasing for all $\beta<\xi<\alpha$, then there must exist a point $\gamma \in(\beta, \alpha)$ such that $h_1(\gamma)=0$ and $h_2(\gamma)>0$. Since $0<x_1<\xi_0$ and $h_4(x_1)<0$, by equation \eqref{eqn:3-4}, then $h_4(x_1)=h_1(x_1)h_2(x_1)-g(x_1)h_3(x_1)=h_1(x_1)h_2(x_1)<0$. Since $h_2(x_1)>0$, then $h_1(x_1)<0$. Hence, it is obvious that $x_1<\gamma<\alpha$ and $h_3(\gamma)<0$. Since $h_1(\xi)<0$ for all $0<\xi<\gamma$ and $h_1(\xi)>0$ for all $\gamma<\xi<\zeta$, then  point $\gamma$ is the minimum value of $g(\xi)$ on $[0, \zeta]$.    Now let's decide the sign of $g(\gamma)$. Since $x_1<\xi_0<\alpha$,  two cases will arise:  
	   
	       $1)$  If $x_1<\gamma<\xi_0$, by equation  \eqref{eqn:3-4}, then $h_4(\gamma)=h_1(\gamma)h_2(\gamma)-g(\gamma)h_3(\gamma)=-g(\gamma)h_3(\gamma)<0$,  then $g(\gamma)<0$. Therefore, in this case, there exists a solution of (\ref{ode1}) and (\ref{bc}).  We will designate this solution as type $II$ solution.

              $2)$  If $\xi_0<\gamma<\alpha$, by equation  \eqref{eqn:3-4}, then $h_4(\gamma)=h_1(\gamma)h_2(\gamma)-g(\gamma)h_3(\gamma)=-g(\gamma)h_3(\gamma)>0$, then $g(\gamma)>0$. Therefore, in this case, there exists a solution of (\ref{ode1}) and (\ref{bc}).  We will designate this solution as type $III$ solution.
\end{proof}

\section{Numerical multiple solutions}
\label{num sol}

Numerical solutions to  equation (\ref{ode1}) subject to  (\ref{bc}) are considered  for all non-negative Reynolds number, $0\leq R<\infty$. All the numerical results are based on the collocation method (e.g. MATLAB boundary value problem solver bvp4c in which we set the relative error tolerance $10^{-10}$).  The numerical results  are shown  in \cref{fig:multiplesolutions} with a plot of skin friction  at the lower wall $-f''(-1)$ versus  Reynolds number $R$.  It can be seen that  in the range of $14.10<R<\infty$ there are three types of solutions for each value of $R$, while only single solution is observed for $0\leq R<14.10$. The solution curves have been labeled $I$, $II$ and $III$ (which  correspond to the three theoretical solutions  in \cref{sec:existence}) suggesting three completely different types of solutions.  For the symmetric flow in a channel, although there are  three types of solutions, two of them have only an exponentially small difference when $R$ is large \cite{robinson1976existence,brady1981steady}.

 \begin{figure}[tbhp] 
\centering
\includegraphics[width=0.4\textwidth,height=0.3\textwidth]{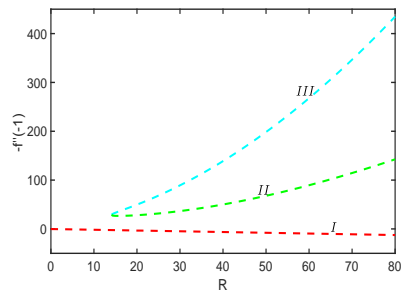}
\caption{ Skin friction at the lower wall versus crossflow Reynolds number for the asymmetric flow through a uniformly porous channel.}  
\label{fig:multiplesolutions}
\end{figure}

Typical velocity profiles of type $I$ solution, i.e. $v\sim f(y)$ and $u\sim f'(y)$,  are shown in \cref{type1vu}. As can be seen from \cref{type1u},  the flows form a  thin boundary layer structure near the lower wall of the channel for the relatively high Reynolds number. The increasing Reynolds number has little influence on the flow character, but the boundary layer is thinner and thinner with the increasing $R$.

\begin{figure}[tbhp]
\centering
\subfloat[Transverse velocity of type $I$]{\includegraphics[width=0.4\textwidth]{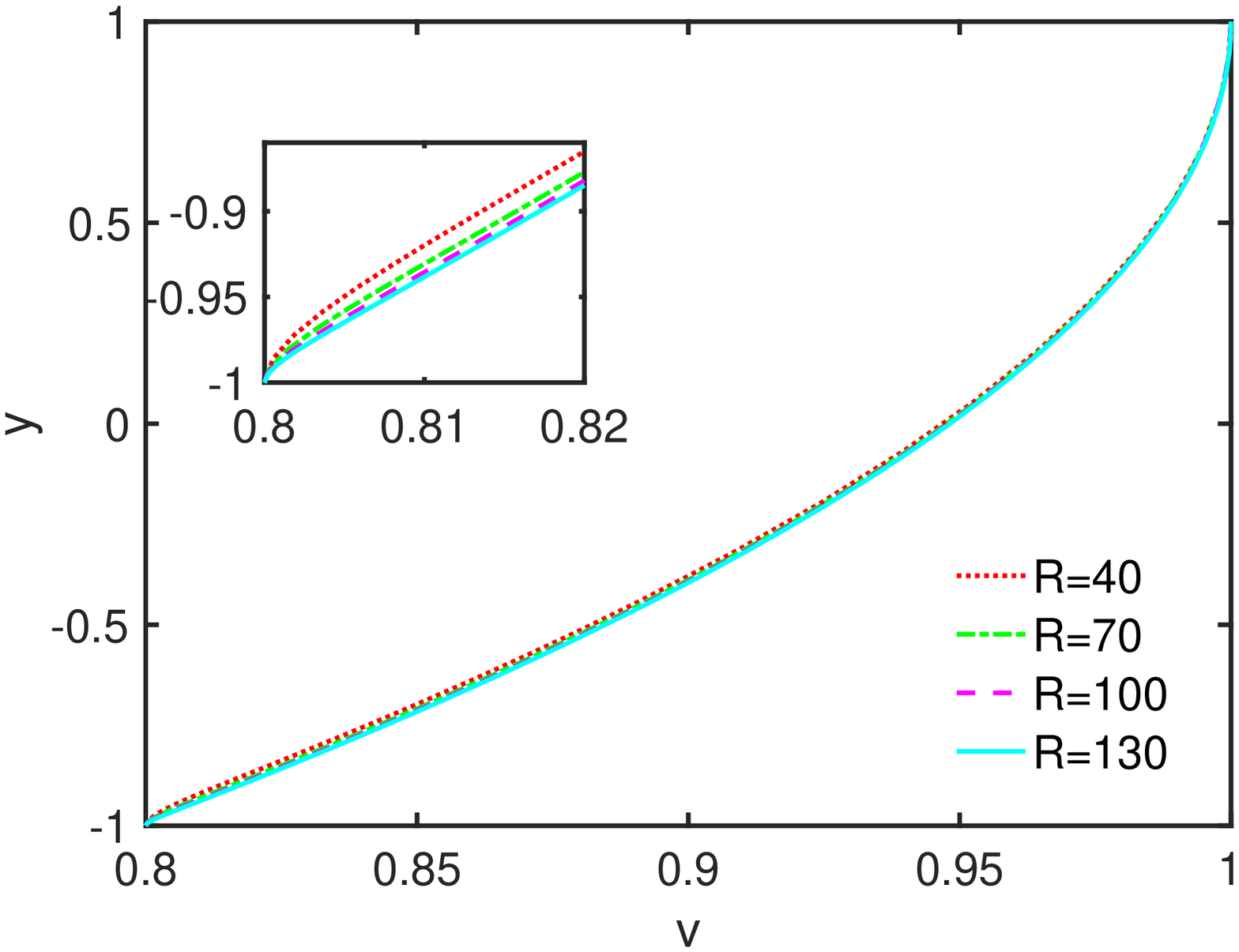}}
\subfloat[Streamwise velocity of type $I$]{\includegraphics[width=0.4\textwidth]{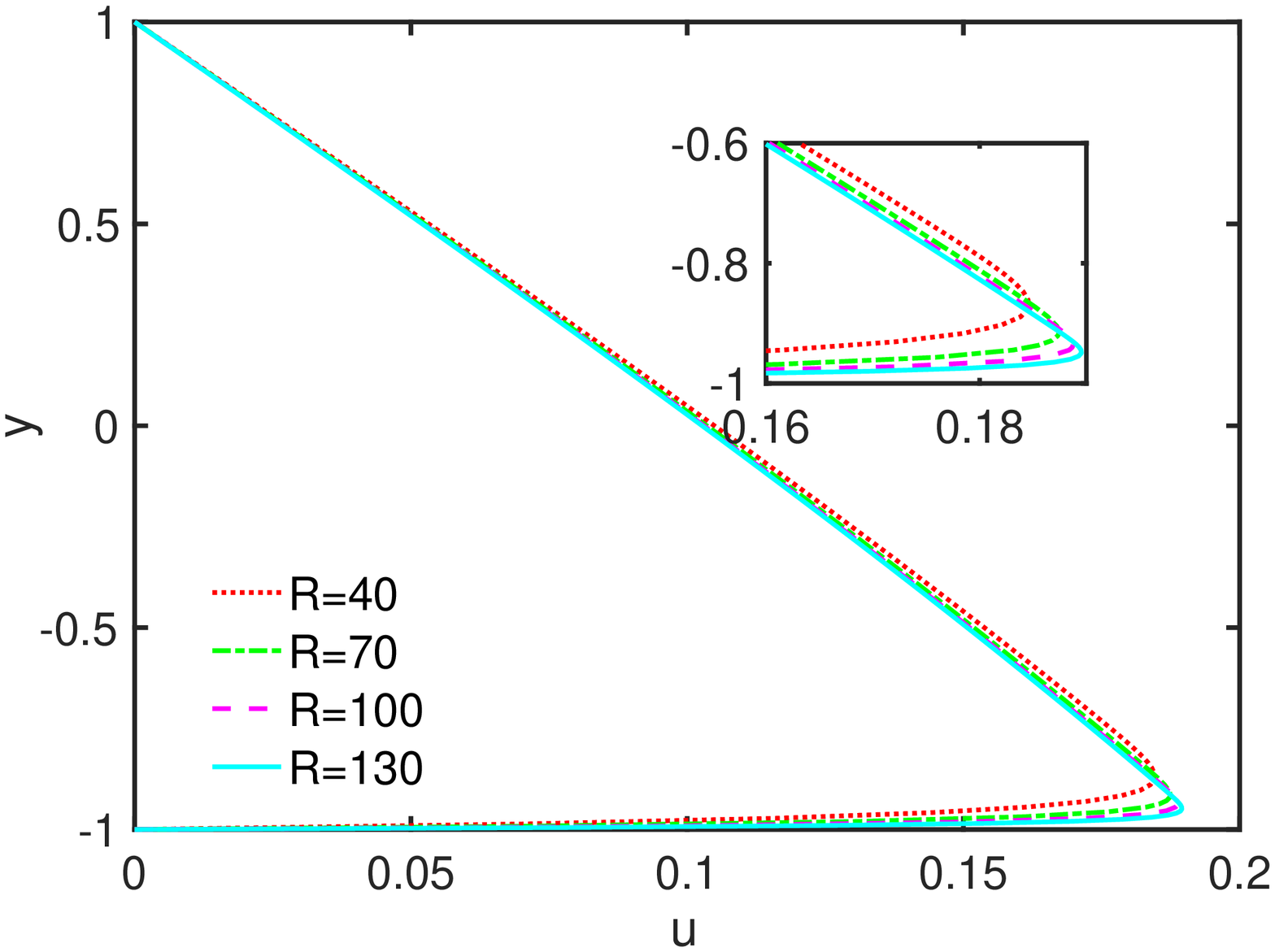}\label{type1u}}
\caption{Velocity profiles of type $I$  at $a=0.8$ with $R=40$,  $R=70$, $R=100$ and $R=130$.} 
\label{type1vu}
\end{figure}

Typical velocity profiles for type $II$ solution  are presented in \cref{type2vu}. All of these flows occur as $R>14.10$.  As $R$ is increased, the minimum  of  transverse velocity in the reverse region is decreasing and the turning points which are the points such that $f(y)=0$ are moving towards the walls of the channel.   The maximum  of streamwise velocity is increasing and the minimum is decreasing with the increasing $R$. 

\begin{figure}[tbhp]
\centering
\subfloat[Transverse velocity of type $II$]{\includegraphics[width=0.4\textwidth]{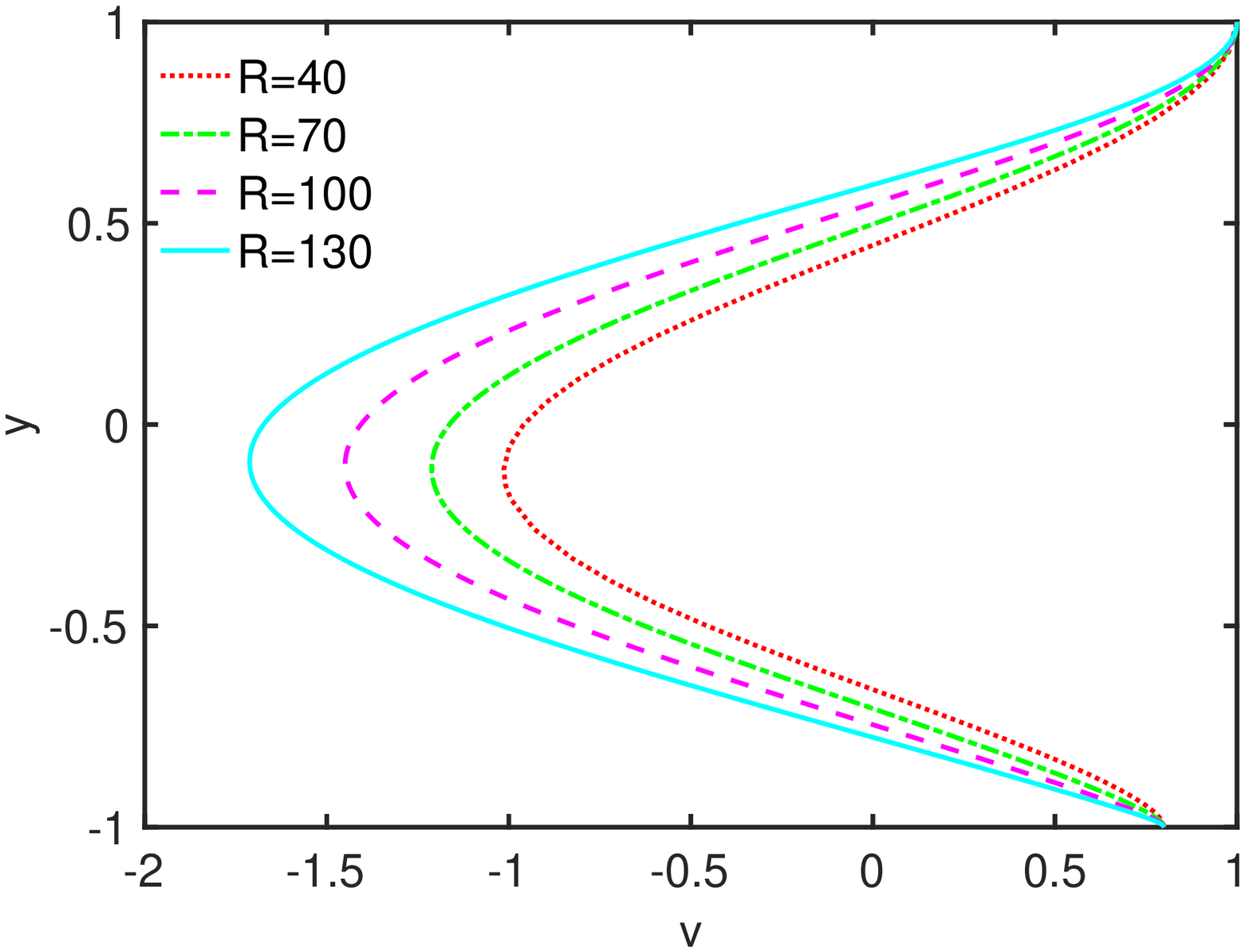}\label{type2v}} 
\subfloat[Streamwise velocity of type $II$]{\includegraphics[width=0.4\textwidth]{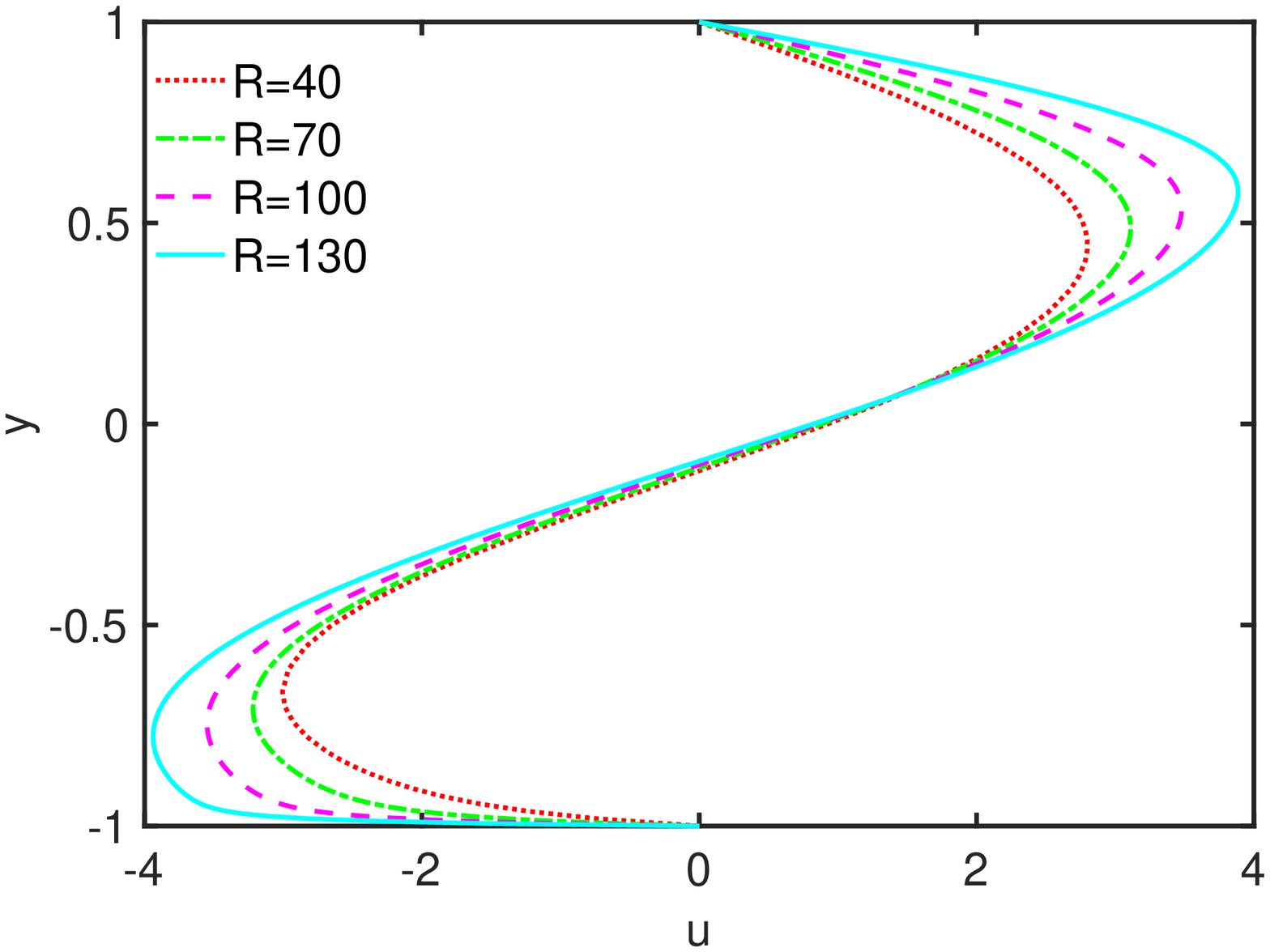}}
\caption{Velocity profiles of type $II$  at $a=0.8$ with $R=40$,  $R=70$, $R=100$ and $R=130$.} 
\label{type2vu}
\end{figure}

The type $III$ solutions, shown in \cref{type3vu}, have  an unusual shape. The rapid decay occurs not only for the streamwise velocity but also for the transverse velocity near the lower wall. With the increasing $R$, the region between the lower wall and the minimum velocity become thinner. There is a region of reverse flow near the lower wall for the streamwise velocity.

\begin{figure}[tbhp]
\centering
\subfloat[Transverse velocity of type $III$]{\includegraphics[width=0.4\textwidth]{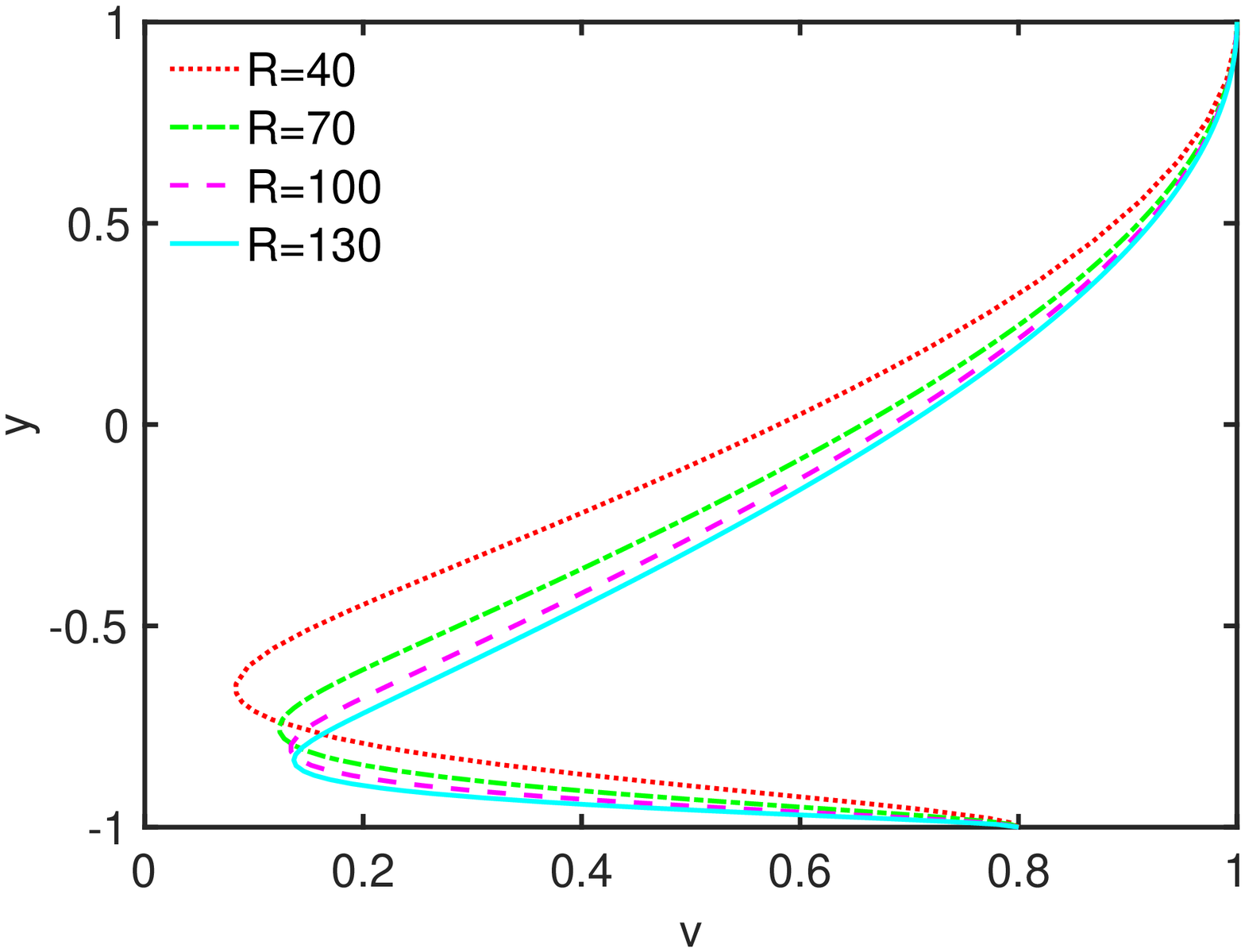}}
\subfloat[Streamwise velocity of type $III$]{\includegraphics[width=0.4\textwidth]{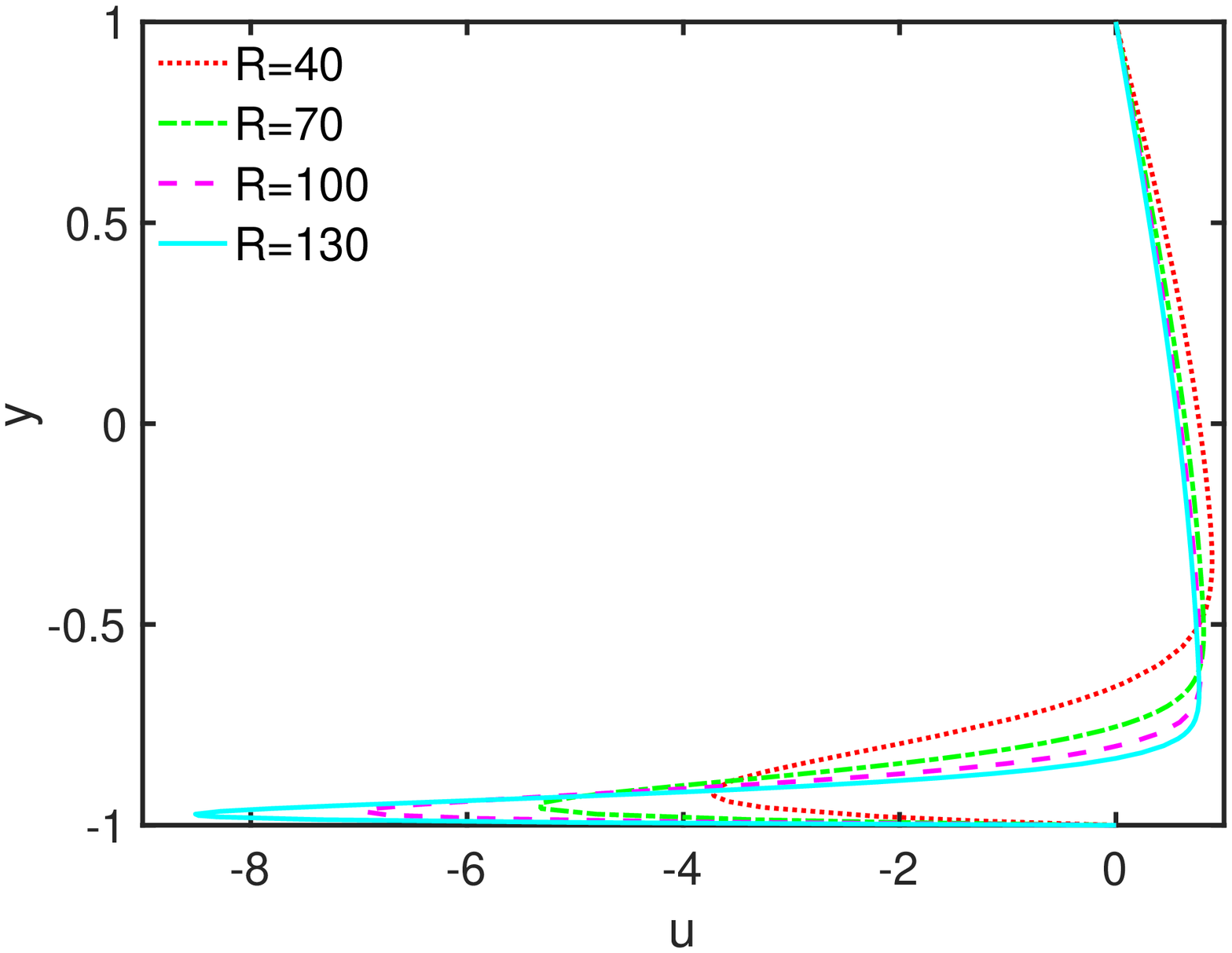}}
\caption{Velocity profiles of type $III$  at $a=0.8$ with $R=40$,  $R=70$, $R=100$ and $R=130$.} 
\label{type3vu}
\end{figure}

All numerical results indicate that the solutions consist of inviscid solution and boundary layer solution which is confined to the viscous layer near the lower wall of the channel. It is obvious that  the flow direction of streamwise velocity inside the boundary layers for type $II$ and type $III$ is opposite to the type $I$.  The reversal flow occurs for both type $II$ and type $III$.

In an effort  to develop a better understanding of the flow character,  we show in \cref{streamlines} sketches of the streamlines to describe the flow behaviour corresponding to the different branches of solutions. These graphs depict all three types of solutions and enable us to deduce their fundamental characteristics. 
\begin{figure}[H]
\centering
\subfloat{\includegraphics[width=0.8\textwidth,height=0.2\textwidth]{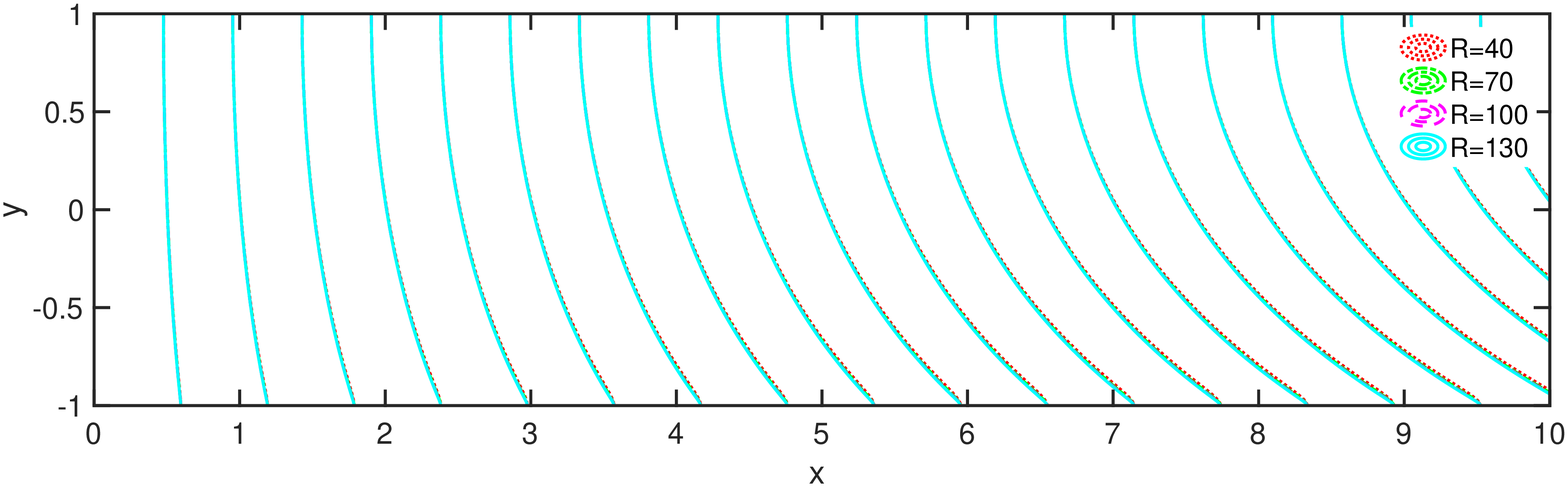}}\\
\subfloat{\includegraphics[width=0.8\textwidth,height=0.2\textwidth]{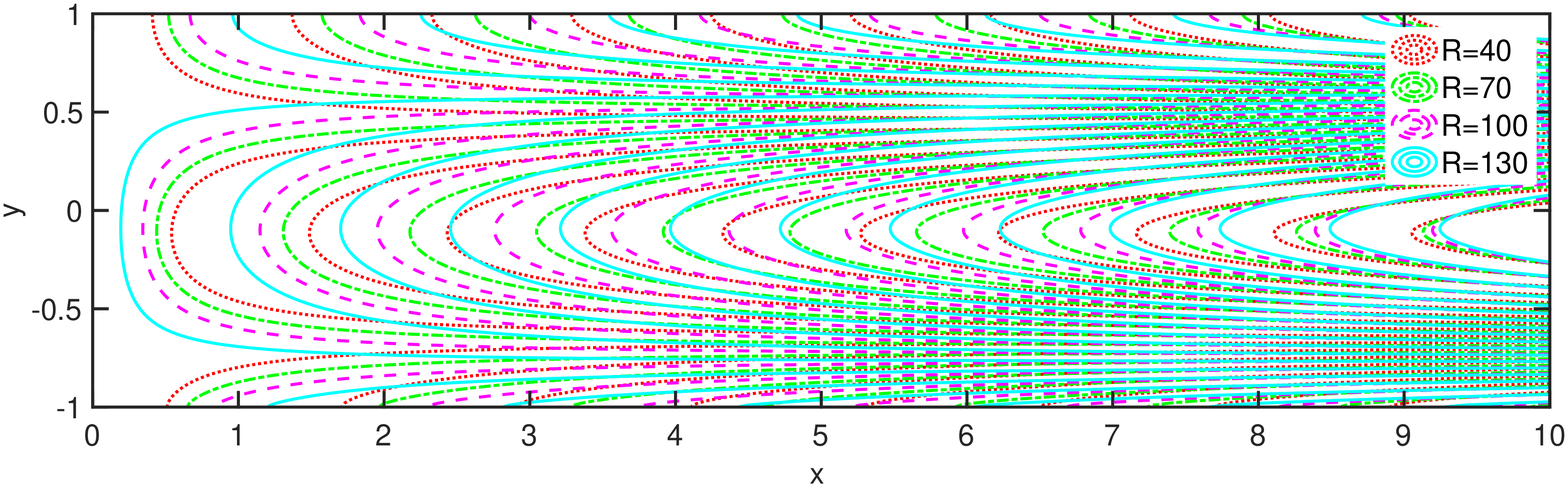}}\\
\subfloat{\includegraphics[width=0.8\textwidth,height=0.2\textwidth]{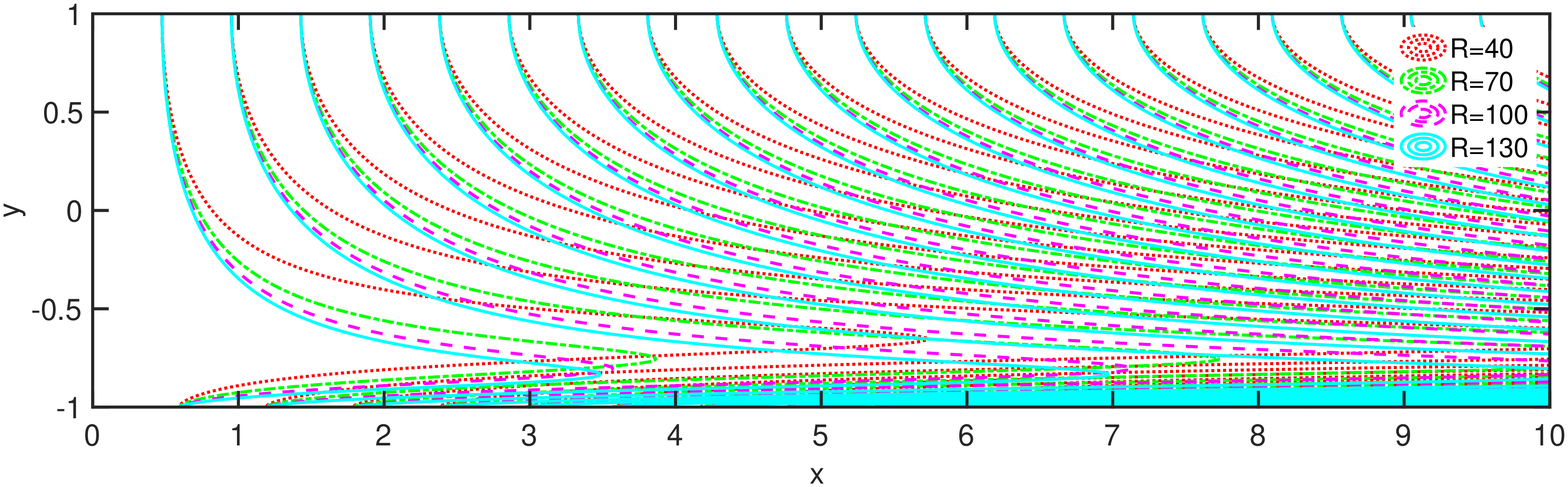}}
\caption{Streamline patterns of types $I$, $II$ and $III$ solutions from top to bottom at $a=0.8$  with  $R=40$, $R=70$, $R=100$ and $R=130$.} 
\label{streamlines}
\end{figure}

\section {Asymptotic multiple solutions for high Reynolds number $R$}
\label{asy sol}
We have shown the existence of multiple solutions and from the numerical solutions we  know that when $R$ is relatively large, there exists three solutions. Since the upper wall is with injection while the lower wall is with suction which indicates that the flow may exhibit a  boundary layer structure near the lower wall for high Reynolds number,  it is of considerable theoretical interest to construct asymptotic solution for the three types solutions which can help us to develop a better understanding of  the characteristics of boundary layer.  
 
By treating $\epsilon=\frac{1}{R}$  as a small perturbation parameter, equation (\ref{ode1}) can be written as 
\begin{equation}
\epsilon f'''+(ff''-f'^2)=k,   \label{odef}
\end{equation}
where $k=K/R$.

\subsection{Asymptotic solution of type $I$}
From the numerical solution of type $I$ in \cref{type1vu}, we can see that the streamwise velocity rapidly decays near the lower wall ($y=-1$). Hence, by the method of boundary layer correction,  $f(y)$ and $k$ can be expanded as follows
\begin{align}
&f(y)=f_0(y)+\epsilon(f_1(y)+h_1(\eta))+\epsilon^2(f_2(y)+h_2(\eta))+\cdots,      \label{fexp}\\
&k=k_0+\epsilon k_1+\epsilon^2 k_2+\cdots,   \label{kexp}
\end{align}
where $\eta=\frac{1+y}{\epsilon}$ is a stretching transformation near   $y=-1$ and $h_i(\eta)$,  $i=1, 2\cdots$ are boundary layer functions.
By substituting (\ref{fexp}) into (\ref{bc}) and collecting the equal power of $\epsilon$, the  boundary conditions become
\begin{align}
&f_0|_{y=1}=1,\quad f'_{0}|_{y=1}=0, \quad f_0|_{y=-1}=a,    \label{bcf0}\\
&f'_{i-1}|_{y=-1}+\dot{h}_i|_{\eta=0}=0,\quad i=1,2,\cdots,  \label{bchi}\\
&f_i|_{y=1}=0,\quad f'_{i}|_{y=1}=0,\quad f_i|_{y=-1}+h_i|_{\eta=0}=0,\quad i=1,2,\cdots,  \label{bcfi} 
\end{align}
where $\dot{h}_i$ denotes the derivative of $h_i$ with respect to $\eta$.  We note here that $f_0(\eta)$ is the solution of the reduced problem
\begin{equation}
f_0 f_0''-f_0'^2=k_0  \label{f0}
\end{equation}
satisfying boundary conditions (\ref{bcf0}).\\
The construction is similar to that of subsection $4.1$ in \cite{hongxia}, where additional factors such as a magnetic force and a boundary expansion rate are considered. So we omit the details here and only provide the asymptotic solution of $(\ref{ode1})$ and (\ref{bc})  for type $I$ solution 
\begin{align}
f(y)=&\cos z+\epsilon \{(Q(z)+b)\sin z +\frac{\lambda}{2b} (z \sin z +\cos z)+  \frac{\lambda}{2b}+\frac{b}{2}(z {\tan}^2 z- \tan z)
{}\nonumber \\
&{}+ \frac{b}{2}(\ln(1-\sin z)- \ln \cos z)(z \sin z+\cos z)+\frac{b}{a}\sin2b \cdot e^{-a\eta}\}+O(\epsilon^2), \label{type1}
\end{align}
where $\eta=\frac{1+y}{\epsilon}$, $z=by-b$, $b=\frac{\cos^{-1}a}{2}$, $Q(z)= b\int_{0}^{z}  \phi \sec \phi(1- \sec^2 \phi )\mathrm{d}\phi$ and
\begin{align}
 \lambda = & \frac{1}{2 a (b \sin{(2b)}+\cos{(2b)})}( 2(b-ab-a  Q(-2b))\sin{(2b)}  +a b \tan {(2b)} \nonumber \\
 &-2 a b^2\tan^2{(2b)}+a b(\cos{(2b)}+2 b \sin{(2b)}) (\ln(1+\sin{(2b)})-\ln\cos{(2b)})). \nonumber
\end{align}

Terrill \cite{terrill1966laminar} considered a similar case where the lower wall is with injection and the upper wall is with suction and  constructed an asymptotic solution of the type $I$ solution  with the method of matched asymptotic expansion. 
of the construction of type $I$ solution, the details are similar to those of Subsection $4.1$ in \cite{hongxia}.

\subsection{Asymptotic solution of type $II$}
Constructing  an asymptotic expansion as $R\rightarrow\infty$ for the solution of type $II$ is  a more complicated process than that presented in the previous subsection. From the numerical solution of type $II$ in \cref{type2vu}, we know that $f(y)$ vanishes at exactly two points $y_1$ in $(-1,0)$ and $y_2$ in $(0,1)$ (called turning points  \cite{macgillivray1994asymptotic}). The technique used in this section follows the symmetric flow case in \cite{lu1997asymptotic,lu1999matched,macgillivray1994asymptotic} where there exists only one turning point. 

Define that the distance between $y=-1$ and $y=y_1$ is $\Delta_1$ and the distance between $y=1$ and $y=y_2$ is $\Delta_2$, hence, it follows that $y_1=-1+\Delta_1$ and $y_2=1-\Delta_2$ which are unknown a priori.
By differentiating $(\ref{odef})$, we obtain 
\begin{equation}
\epsilon f^{iv}+(ff'''-f'f'')=0.  \label{ode4f}
\end{equation}
$\mathbf{1)}$ Asymptotic solution between the turning points $y_1$ and $y_2$

Letting $\epsilon=0$,   equation (\ref{ode4f}) become
\begin{equation}
ff'''-f'f''=0. 
\end{equation}
We observe three types of solutions for the equation: $cy$, $c\sinh (dy+e)$ and $c\sin(dy+e)$.  But, to have the solution be valid uniformly in $[y_1, y_2]$  and satisfy the conditions $f(y_1)=0$ and $f(y_2)=0$, the following has to hold:
\begin{subequations}
\begin{align}
f(y)&\sim \Lambda\sin \frac{\pi}{2-\Delta_1-\Delta_2}(y-(-1+\Delta_1)) \label{fleft} \\
&= -\Lambda\sin \frac{\pi}{2-\Delta_1-\Delta_2}(y-(1-\Delta_2)),  \label{fright}
\end{align}
\end{subequations}
where $\Lambda<0$ is a constant. \cref{type2v} shows that the turning points $y_1$ and $y_2$ are moving towards the left-end point and the right-end point of the interval $[-1, 1]$, respectively, with  increasing $R$. The quantities $\Delta_1$, $\Delta_2$ and $\Lambda$ which are related to $\epsilon$, will be determined by matching  as $\epsilon\rightarrow0$.\\
$\mathbf{2)}$ Asymptotic solution near $y=y_1$ and inner solution near $y=-1$ 

We introduce a variable  transformation 
\begin{equation}
\tau=\frac{-1+\Delta_1-y}{\Delta_1}, \quad y\in[-1, -1+\Delta_1].
\end{equation}
Letting $f(y)=f(-1+\Delta_1-\tau\Delta_1)=\overline{f}(\tau)$, then,  (\ref{ode4f}) becomes 
\begin{equation}
\overline{\epsilon}\overline{f}^{iv}-(\overline{f}\overline{f}'''-\overline{f}'\overline{f}'')=0,  \label{olf}
\end{equation}
where $\overline{\epsilon}=\frac{\epsilon}{\Delta_1}$. The boundary conditions to be satisfied by (\ref{olf}) are
\begin  {equation}
\overline{f}(0)=0,\quad  \overline{f}(1)=a, \quad \overline{f}'(1)=0. \label{olbc}
\end{equation}
Since $\overline{\epsilon}\rightarrow0$ as $\epsilon\rightarrow0$, (\ref{olf}) subject to (\ref{olbc})  is still a singular perturbation problem. \\
$\mathbf{(1)}$ Outer solution

Setting $\overline{\epsilon}=0$, the reduced equation is 
\begin{equation}
\overline{f}\overline{f}'''-\overline{f}'\overline{f}''=0, \label{f-}
\end{equation}
satisfying the boundary condition $\overline{f}(0)=0, \overline{f}(1)=a$ and $\overline{f}'(\tau)>0$ for all $\tau$. Equation (\ref{f-}) may have three possible solution:  $\sigma \tau$, $a \sin\frac{\pi}{2} \tau$ and $a \sinh (\ln (\frac{1+\sqrt{5}}{2})\tau)$. By the proof of  \cref{prop:3-7}(c)  for the type $II$ solution, we know that $\frac{1}{2}b(y_1+1)<\gamma<\xi_0$, then $g^{iv}(\xi)<0$ in $(0, \frac{1}{2}b(y_1+1))$, thus $\overline{f}^{iv}(\tau)<0$ in $(0,  1)$.  Hence, trigonometric functions and hyperbolic functions can be excluded. The outer solution is  
\begin{equation}
\overline{f}(\tau)=\sigma \tau+\cdots,
\end{equation}
where $\sigma$ can be  determined  by matching.\\
$\mathbf{(2)}$ Inner solution

The lower wall of the channal is with suction, hence, we introduce a stretching variable $x^{*}=\frac{1-\tau}{\overline{\epsilon}}$. 
Letting $\overline{f}(\tau)=\widehat{f}(x^{*})$, then, (\ref{olf}) becomes 
\begin{equation}
\widehat{f}''''+\widehat{f}\widehat{f}'''-\widehat{f}'\widehat{f}''=0.   \label{wdf}
\end{equation}
The conditions at point $x^*=0$ are $\widehat{f}(0)=a$ and $\widehat{f}'(0)=0$.\\
The inner solution can be expanded as:
\begin{equation}
\widehat{f}(x^*)=a+\overline{\epsilon}\widehat{f}_1(x^*)+\cdots.         \label{wdfexp}
\end{equation}
Substituting (\ref{wdfexp}) into (\ref{wdf}) and collecting the terms of $O(\overline{\epsilon})$, we can obtain the equation of $\widehat{f}_{1}(x^*)$
\begin{equation}
\widehat{f}_1^{iv}+a\widehat{f}_1'''=0,    \label{wdf1}
\end{equation}
satisfying $\widehat{f}_1(0)=\widehat{f}_1'(0)=0$. Then, the expression of $\widehat{f}_{1}(x^*)$ is
\begin{equation}
\widehat{f}_1(x^*)=b_1(e^{-ax^*}+ax^*-1)+b_2{x^*}^2,
\end{equation}
where $b_1$ and $b_{2}$ will be determined by matching.
Hence, the inner solution becomes 
\begin{equation}
\widehat{f}(x^*)=a+\overline{\epsilon}(b_1(e^{-ax^*}+ax^*-1)+b_2{x^*}^2)+\cdots. \label{wdfx*}
\end{equation}
Meanwhile, assume that the expression of outer solution can be written as
\begin{equation}
\overline{f}(\tau)=\sigma \tau+\overline{\epsilon}\overline{f}_1(\tau)+\overline{\epsilon}^2\overline{f}_2(\tau) +\cdots.  \label{olftau}
\end{equation}
Substituting $(\ref{olftau})$ into $(\ref{olf})$ yields that $\overline{f}_1(\tau)$ satisfies 
\begin{equation}
\tau\overline{f}_1'''-\overline{f}_1''=0.
\end{equation}
The correspongding condition is $\overline{f}_1(0)=0$. Then, the expression of $\overline{f}_1(\tau)$ is 
\begin{equation}
\overline{f}_1(\tau)=\frac{1}{6}c_1\tau^3+d_1\tau,
\end{equation}
where $c_{1}$ and $d_{1}$ are constants.
The outer solution (\ref{olftau}) expressing in terms of inner variable $x^{*}$ is
\begin{align}
\overline{f}(\tau)&=\sigma(1-\overline{\epsilon}x^*)+\overline{\epsilon}(c_1(1-\overline{\epsilon}x^*)^3+d_1(1-\overline{\epsilon}x^*))+\cdots {}\nonumber\\
{}&=\sigma+\overline{\epsilon}(-\sigma x^*+c_1+d_1)+\cdots.    \label{olfx*}
\end{align}

\noindent Matching the inner solution (\ref{wdfx*}) with the outer solution (\ref{olfx*}) gives $\sigma=a$, $b_1=-1$, $b_2=0$ and  $c_1+d_1=0$.
Following the analysis in \cite{lu1999matched}, we know that $\overline{f}_{i}(\tau)$ are all linear, $i=1,2,\cdots$, where $c_{1}=0$ and $d_{1}=1$. Hence, $\overline{f}(\tau)$ can be written as 
\begin{equation}
\overline{f}(\tau)\sim\theta(\overline{\epsilon})\tau,
\end{equation}
where $\theta(\overline{\epsilon})=a+d_1\overline{\epsilon}+d_2\overline{\epsilon}^2+\cdots$ and $\theta(\overline{\epsilon})\rightarrow a$ as $\overline{\epsilon}\rightarrow0$.
 The inner solution has exponentially small terms and outer solution has to be more precise, we assume that $\overline{f}(\tau)$ is as follow:
\begin{equation}
\overline{f}(\tau)=\theta(\overline{\epsilon})\tau+\sum_{i=1}^\infty\delta_ih_i(\tau),  \label{olfh}
\end{equation}
where $\delta_i=\delta_i(\overline{\epsilon})=o(\overline{\epsilon}^n)$ and $\delta_{i+1}<<\delta_i$ for all positive integers $n$ and $i$.
(\ref{olfh}) is valid in the small neighborhood of the turning point $y_1=-1+\Delta_1$.
Substituting (\ref{olfh}) into (\ref{olf}) and collecting the terms of  $O(\delta_{1})$ yield
\begin{equation}
\overline{\epsilon}h_1^{iv}-\theta\tau h_1'''+\theta h_1''=0,
\end{equation}
satisfying the condition $h_{1}(0)=0$. Following the similar analysis \cite{macgillivray1994asymptotic}, one  solution of $h_1$ is $h_{1}(\tau)=\frac{1}{6}\tau^{3}+r_{1}\tau$, where $r_1$ is a constant.
Setting $\delta_{2}=\delta_{1}^{2}$ and collecting the terms of  $O(\delta_{2})$ yield
\begin{equation}
\overline{\epsilon}h_2^{iv}-\theta\tau h_2'''+\theta h_2''+ \frac{\tau^3}{3}=0. \label{olh2}
\end{equation}
Differentiate (\ref{olh2}) and multiply by the integrating factor $e^{{-\frac{\theta}{2\overline{\epsilon}}\tau^{2}}}$,  then, we can obtain
\begin{equation}
h_2^{iv}= -\frac{1}{\overline{\epsilon}}e^{\frac{\theta}{2\overline{\epsilon}}\tau^2}\int_0^r s^2 e^{-\frac{\theta}{2\overline{\epsilon}}s^2}ds-Ce^{\frac{\theta}{2\overline{\epsilon}}\tau^2}, \label{olh24}
\end{equation}
where $C$ is a constant. If we choose $\tau<0$ which is away from zero, $h_{2}^{iv}$ will have exponentially large term. Then, we can choose  $C$ to eliminate the exponentially large term. Evaluating (\ref{olh24}) leads to 
\begin{equation}
h_2^{iv}= -\frac{1}{\overline{\epsilon}}e^{\frac{\theta}{2\overline{\epsilon}}\tau^2}\{-(\frac{2\overline{\epsilon}}{\theta})^{3/2} [\frac{\sqrt{\pi}}{4}-\frac{1}{2}\sqrt{\frac{\theta}{2\overline{\epsilon}}}\vert\tau\vert e^{-\frac{\theta}{2\overline{\epsilon}}\tau^2}+\cdots]\}-Ce^{\frac{\theta}{2\overline{\epsilon}}\tau^2}.   \label{olh24evl}
\end{equation}
Hence, we choose $C=\frac{\sqrt{2\overline{\epsilon}\pi}}{2\theta\sqrt{\theta}}$.
Evaluating (\ref{olh24evl}),  we obtain asymptotic expression 
\begin{equation}
h_2^{iv}\sim\theta^{-1}\tau.
\end{equation}
Hence, the expression for $\tau<0$  is
\begin{equation}
\overline{f}(\tau)=\theta\tau+\delta_1(\frac{\tau^3}{6}+r_1\tau)+\delta_1^2\theta^{-1}(\frac{\tau^5}{5!}+\cdots)+\cdots. \label{olfin}
\end{equation}
Then, expanding  (\ref{fleft}) at the turning point $y_{1}=-1+\Delta_{1}$ yields
\begin{align}
f(y)&\sim \Lambda\sin\frac{\pi}{2-\Delta_1-\Delta_2}(y-(-1+\Delta_1))       {}\nonumber\\
 &{}=-\Lambda\frac{\pi\Delta_1\tau}{2-\Delta_1-\Delta_2}+\frac{\Lambda}{3!}(\frac{\pi\Delta_1\tau}{2-\Delta_1-\Delta_2})^3-\frac{\Lambda}{5!}(\frac{\pi\Delta_1\tau}{2-\Delta_1-\Delta_2})^5+\cdots. \label{fout}
\end{align}

\noindent Comparing the linear term in (\ref{olfin}) and (\ref{fout}),  we can obtain 
\begin{equation}
\Lambda\sim-\frac{a(2-\Delta_1-\Delta_2)}{\pi\Delta_1}, \label{A1}    
\end{equation}
where $\theta\sim a$ is used.\\
Then, comparing the cubic term, we get 
\begin{equation}
\delta_{1} \sim -a(\frac{\pi\Delta_{1}}{2-\Delta_{1}-\Delta_{2}})^{2}.
\end{equation}
Hence, the asymptotic expansion of $\overline{f}(\tau)$ is
\begin{equation}
\overline{f}(\tau)=\theta\tau-a(\frac{\pi\Delta_1}{2-\Delta_1-\Delta_2})^2\frac{\tau^3}{6}+a^2(\frac{\pi\Delta_1}{2-\Delta_1-\Delta_2})^4\theta^{-1}h_2+\cdots.   \label{olff}
\end{equation}
$\mathbf{3)}$ The determination of  $\Delta_1$  and  $\Delta_2$

In this section, we will find the asymptotic relationship between $\Delta_{1}$,  $\Delta_{2}$  and $\epsilon$ by matching near $\tau=1$. From  (\ref{olh24}) with $C=\frac{\sqrt{2\overline{\epsilon}\pi}}{2\theta\sqrt{\theta}}$, we can know  
\begin{equation}
h_2^{iv}= -\frac{1}{\overline{\epsilon}}e^{\frac{\theta}{2\overline{\epsilon}}\tau^2}\int_0^r s^2 e^{-\frac{\theta}{2\overline{\epsilon}}s^2}ds-\frac{\sqrt{2\overline{\epsilon}\pi}}{2\theta\sqrt{\theta}}e^{\frac{\theta}{2\overline{\epsilon}}\tau^2}=-\frac{\sqrt{2\overline{\epsilon}\pi}}{\theta\sqrt{\theta}}e^{\frac{\theta}{2\overline{\epsilon}}\tau^2}+\frac{1}{\theta}\tau+\cdots.   \label{olh2f}
\end{equation}
Then, from (\ref{olff}) and (\ref{olh2f}), we can have 
\begin{align}  \label{olfmattau}
\frac{d^4\overline{f}}{d\tau^4}&=\delta_1^2h_2^{iv}(\tau) \\
&{}=-a^{1/2}\pi^{9/2}\frac{\Delta_1^{7/2}}{(2-\Delta_1-\Delta_2)^4}\sqrt{2\epsilon}e^{\frac{\theta}
{2\overline{\epsilon}}\tau^2}+a^2(\frac{\pi\Delta_1}{2-\Delta_1-\Delta_2})^4\frac{1}{\theta}\tau+\cdots. \nonumber
\end{align}
The outer solution (\ref{olfmattau}) expressing in the terms of inner variable $x^{*}$ is 
\begin{align}  \label{olfmatx*}
\frac{1}{\overline{\epsilon}^4}\frac{d^4\overline{f}}{dx^{*4}}=&-a^{1/2}\pi^{9/2}\frac{\Delta_1^{7/2}}{(2-\Delta_1-\Delta_2)^4}\sqrt{2\epsilon}e^{\frac{\theta\Delta_1}{2\epsilon}}e^{\frac{\theta\overline{\epsilon}}{2}x^{*2}}e^{-\theta x^*} \nonumber \\  
&+a^2(\frac{\pi\Delta_1}{2-\Delta_1-\Delta_2})^4\frac{1}{\theta}(1-\overline{\epsilon}x^*)+\cdots.            
\end{align}
Differentiate (\ref{wdfx*}) four times gives
\begin{equation}
\frac{1}{\overline{\epsilon}^4}\frac{d^4\overline{f}}{dx^{*4}}=-\frac{a^4}{\overline{\epsilon}^3}e^{-ax^*}+\cdots.\label{wdfmatx*}
\end{equation}
Comparing (\ref{olfmatx*}) and (\ref{wdfmatx*}) suggests that  the overlap domain must satisfy the conditions: $\frac{\overline{\epsilon}\theta}{2}x^{*2}<<1$ and $x^{*2}>>1$.
 It is obvious that
\begin{equation}
-a^{1/2}\pi^{9/2}\frac{\Delta_1^{7/2}}{(2-\Delta_1-\Delta_2)^4}\sqrt{2\epsilon}e^{\frac{\theta\Delta_1}{2\epsilon}}\sim-\frac{a^4}{\overline{\epsilon}^3}.
\end{equation}
Finally, setting $\theta\sim a+\overline{\epsilon}$, we obtain the asymptotic relationship:
\begin{equation}
\frac{\Delta_1}{\epsilon}e^{a\frac{\Delta_1}{\epsilon}}= \frac{a^7(2-\Delta_1-\Delta_2)^8}{2e\pi^9\epsilon^8}.  \label{Delta1}
\end{equation}
$\mathbf{4)}$ Asymptotic solution  near $y=y_2$

In order to analyze the asymptotic behaviour  near $y=y_2$, we also introduce a  variable transformation
\begin{equation}
\eta=\frac{y-1+\Delta_2}{\Delta_2},   \quad  y\in[1-\Delta_2, 1].
\end{equation}
Letting $f(y)=f(1-\Delta_2+\eta\Delta_2)=\tilde{f}(\eta)$, then, (\ref{ode4f}) becomes 
\begin{equation}
\tilde{\epsilon}\tilde{f}^{iv}+(\tilde{f}\tilde{f}'''-\tilde{f}'\tilde{f}'')=0,  \label{tlf}
\end{equation}
where $\tilde{\epsilon}=\frac{\epsilon}{\Delta_2}$. The boundary conditions to be satisfied by (\ref{tlf}) are
\begin  {equation}
\tilde{f}(0)=0, \quad \tilde{f}(1)=1, \quad \tilde{f}'(1)=0, \label{tlbc}
\end{equation}
$\tilde{\epsilon}\rightarrow0$ as $\epsilon\rightarrow0$, but there is no boundary layer near  $\eta=1$ (or $y=1$), hence, (\ref{olf}) and (\ref{olbc})  form a regular perturbation problem.
Setting $\tilde{\epsilon}=0$, the reduced equation is 
\begin{equation}
\tilde{f}\tilde{f}'''-\tilde{f}'\tilde{f}''=0 
\end{equation}
satisfying the boundary condition $(\ref{tlbc})$. 
The corresponding  solution  is 
\begin{equation}
\tilde{f}(\eta)=\sin \frac{\pi}{2}\eta.
\end{equation}
Since there is no boundary layer near the upper wall of the channel,  we expand $\tilde{f}(\eta)$ at the point $\eta=0$
\begin{equation}
\tilde{f}(\eta)=\frac{\pi}{2}-\frac{1}{3!}(\frac{\pi}{2}\eta)^3+\frac{1}{5!}(\frac{\pi}{2}\eta)^5+O(\tilde{\epsilon}).  \label{tlfexp}
\end{equation}
Then, expand (\ref{fright}) at the turning point $y_2=1-\Delta_2$
\begin{align}
f(y)&\sim-\Lambda\sin\frac{\pi}{2-\Delta_1-\Delta_2}(y-(1-\Delta_2))       {}\nonumber\\
     &{}=-\Lambda\frac{\pi\Delta_2\eta}{2-\Delta_1-\Delta_2}+\frac{\Lambda}{3!}(\frac{\pi\Delta_2\eta}{2-\Delta_1-\Delta_2})^3-\frac{\Lambda}{5!}(\frac{\pi\Delta_2\eta}{2-\Delta_1-\Delta_2})^5+\cdots. \label{fyexp}
\end{align}
Comparing the linear term in (\ref{tlfexp}) and (\ref{fyexp}):  $\frac{\pi}{2}\sim-\Lambda\frac{\pi\Delta_2}{2-\Delta_1-\Delta_2}$, then we can obtain 
\begin{equation}
\Lambda\sim-\frac{2-\Delta_1-\Delta_2}{2\Delta_2}. \label{A2}
\end{equation}
From  (\ref{A1}) and (\ref{A2}), the relationship between $\Delta_1$ and $\Delta_2$ is obvious: 
\begin{equation}
\frac{\Delta_{2}}{\Delta_{1}}=\frac{\pi}{2a}.  \label{Delta2}
\end{equation}

\subsection{Asymptotic solution of type $III$}
The numerical solution for type $III$ in \cref{type3vu} shows that, as  $R\rightarrow\infty$, the flow should consist of an inviscid core and a  thin boundary layer near the lower wall.  Both  transverse and streamwise velocities rapidly decay  and then the streamwise velocity  rapidly increases near the lower wall for type $III$ solution while only streamwise velocity rapidly decays for type $I$ and type $II$ solutions. Therefore, it is reasonable to expect that the high Reynolds number structure of the flow can be determined by boundary layer theory near the lower wall. Further, in this case we expect from numerical results that only two boundary conditions at the upper wall $(y=1)$ are satisfied by the reduced problem. This makes the construction much harder than that of type $I$ solution. We expand $k$ as (\ref{kexp}) and  $f$ as follow
\begin{equation}
f(y)=f_0(y)+h_0(\eta)+\epsilon(f_1(y)+h_1(\eta))+\epsilon^2(f_2(y)+h_2(\eta))+\cdots,      \label{fexp3}
\end{equation}
where $\eta=\frac{1+y}{\epsilon}$ is a stretching transformation near the lower wall dimensionless height $y=-1$ and $h_i(\eta)$, $ i=0,1, 2\cdots$ are boundary layer functions. 
By substituting (\ref{fexp3}) into (\ref{bc}), the  boundary conditions become
\begin{align}
&f_0|_{y=1}=1, \quad f'_{0}|_{y=1}=0,   \label{bcf02}\\
&h_0|_{\eta=0}=a-f_0|_{y=-1}, \quad \dot{h_0}|_{\eta=0}=0, \label{bch02}\\
&f_i|_{y=1}=0,\quad f'_{i}|_{y=1}=0,\quad i=1,2,\cdots,  \label{bcfi2} \\
&h_i|_{\eta=0}=-f_i|_{y=-1},\quad \dot{h_i}|_{\eta=0}=-f'_{i-1}|_{y=-1},\quad i=1,2,\cdots,  \label{bchi2}
\end{align}
where $\dot{h_0}$ denotes the derivative of $h_0$ with respect to  $\eta$.
Substituting (\ref{fexp3}) and (\ref{kexp}) into (\ref{odef}) and collecting the terms of $O(1)$, we can obtain the equation of $f_0$ (same as (\ref{f0})):
\begin{equation}
f_0 f_0''-f_0'^2=k_0,  \label{f03}
\end{equation}
satisfying boundary conditions (\ref{bcf02}) (different from (\ref{bcf0})).
\noindent Similarly, collecting the terms of $O(\epsilon^{-2})$, we can obtain the equation of $h_0$:
\begin{equation}
\dddot{h_0}+(h_0+f_0(-1))\ddot{h_0}-\dot{h_0}^2=0,  \label{h0}
\end{equation}
satisfying boundary conditions (\ref{bch02}).

\noindent One expression of  $f_0$ with  the boundary conditions  (\ref{bcf02}) is 
\begin{equation}
f_0=\cos (by-b), \label{f02}
\end{equation}
where $b$ is an undetermined parameter and we denote $f_0(-1)=\cos 2b$ as $\beta$. We shall determine $\beta$ such that equation $(\ref{h0})$ subject to boundary conditions (\ref{bch02}) has a boundary layer solution. Usually we request a boundary layer function to tend to zero as $\eta\rightarrow\infty$. However for problem (\ref{h0}) with (\ref{bch02}) such a solution may not exist. A rigorous proof is highly nontrivial, we will report it in a forthcoming  paper. For the purpose of the construction of the first order asymptotic solution here in the paper, it is enough to request a boundary layer function $h_0(2/\epsilon)\rightarrow0$ (or much smaller than $O(\epsilon)$) when $\epsilon$ is sufficiently small. It is obvious that $h_0(\eta)=a-\beta$ and $(a-\beta)e^{-\beta \eta}$ are two  solutions of $(\ref{h0})$, but the former is not a boundary layer function and the latter doesn't satisfy $ (\ref{bch02})$.  It is hardly possible, however,   to obtain any other explicit solution for the nonlinear equation (\ref{h0})  with  (\ref{bch02}). We thus make use of both analytic and numerical tools  to predict $\beta$. 

Next, we shall show that $\beta<0$ is impossible.
\begin{proposition}\label{lem:1-1}
	Let $h_0(\eta)$ be a boundary layer function solution of  (\ref{h0}) and (\ref{bch02}) in $[0,2/\epsilon)$, then we can have:
	\begin{itemize}
		\item[(a)]  If $h_0''(\eta_1)>0$ for some $\eta_1 \geq0$, then $h_0''(\eta_1)>0$ for all $\eta \geq \eta_1$.
                 \item[(b)] There holds that $h_0'(\eta)\leq0$ for all $\eta\geq0$.                  \item[(c)] There holds that $h_0'(\eta)<0$ for all $\eta>0$.
                 \item[(d)] There exists some $\eta_2>0$ such that $h_0''(\eta)>0$ for all $\eta\geq \eta_2$.
                 \item[(e)]  $\beta<0$ is impossible.
        \end{itemize}
\end{proposition}
	
\begin{proof}
	(a) Let $h_2(\eta)=h_0''(\eta)$ for all $\eta\geq 0$, by equation (\ref{h0}), then $h_2'+(h_0+\beta)h_2=(h_0')^2\geq0$ for all $\eta\geq0$, which implies that $\displaystyle \left(e^{\int_{\eta_1}^\eta (h_0(t)+\beta) dt} h_2(\eta)\right)'\geq0$ for all $\eta\geq \eta_1$. So we get $\displaystyle e^{\int_{\eta_1}^\eta (h_0(t)+\beta) dt} h_2(\eta)\geq h_2(\eta_1)=h_0''(\eta_1)>0$ for all $\eta\geq \eta_1$, which implies that $h_0''(\eta)=h_2(\eta)>0$ for all $\eta\geq \eta_1$.
	
        (b) If not, that is, there exists some $\lambda_0>0$ such that $h_0'(\lambda_0)>0$. Since $h_0'(0)=0$, then there exists some $b\in[0,\lambda_0)$ such that $h_0'(b)=0$ and $h_0'(\eta)>0$ for all $\eta\in(b,\lambda_0]$. Then $h_0''(b)>0$. By the result of part (a), then $h_0''(\eta)>0$ for all $\eta\geq b$, which implies that $h_0'(\eta)$ is increasing in $[b,2/\epsilon)$. So $h_0'(\eta)\geq h_0'(\lambda_0):=\sigma>0$ for all $\eta\geq \lambda_0$, which implies that $h_0(\eta)-h_0(\lambda_0)\geq \sigma(\eta-\lambda_0)$ for all $\eta\geq \lambda_0$. Since $\sigma>0$, by taking $\eta\rightarrow2/\epsilon$ sufficiently large (or $\epsilon$ sufficiently small), then $h_0(2/\epsilon)$ is large, and then is in contradiction to a boundary layer function.  

         (c) If not, by the result of part (b), then there exists some $\lambda_1>0$ such that $\displaystyle 0=h_0'(\lambda_1)=\sup_{\eta\geq0}\ h_0'(\eta)$, which implies that $h_0''(\lambda_1)=0$. By the uniquenes theorem of solution to ODE, then $h_0=h_0(\lambda_1)$ in  $[0,2/\epsilon)$,  is in contradiction to a boundary layer function. 
	
         (d) If not, that is, $h_0''(\eta)\leq0$ for all $\eta\geq0$, then $h_0'(\eta)$ is non-increasing on $[0,2/\epsilon)$. By the result of part (c), then $h_0'(\eta)\leq h_0'(1)<0$ for all $\eta\geq 1$, which implies that $h_0(\eta)-h_0(1)\leq h_0'(1)(\eta-1)$ for all $\eta\geq 1$. Since $h_0'(1)<0$, by taking $\eta\rightarrow2/\epsilon$ sufficiently large, then $h_0(2/\epsilon)$ is negatively large, is  in contradiction to a boundary layer function. 
	
        (e) If not, that is  $\beta<0$, since $h_0(2/\epsilon)$ is sufficiently  small as $\epsilon$ is sufficiently small,  then when $\eta<2/\epsilon$ is sufficiently large (say $\eta> \eta_3> \eta_2$), we have $\beta+h_0(\eta) <0$ for all $\eta > \eta_3$.
From (\ref{h0}), we have 
\begin{equation}
h_0''(\eta)=h_0''(\eta_3)e^{-\int_{\eta_3}^\eta(h_0(t)+\beta)dt}+\int_{\eta_3}^\eta h_0'^2(s)e^{-\int_s^\eta(h_0(t)+\beta)dt}ds. \label{h01}
\end{equation}
We mark the right most term as $B(\eta)$.
Integrating  (\ref{h01}) from $\eta_{3}$ to $\eta$, we can obtain 
\begin{equation}
h_{0}'(\eta)=h_{0}'(\eta_{3})+h_{0}''(\eta_{3}) \int_{\eta_{3}}^{\eta}e^{-\int_{\eta_3}^x(h_{0}(t)+\beta)dt}dx+\int_{\eta_{3}}^{\eta}B(x)dx.   \label{h02}
\end{equation}
Fixed $\eta_{3}$,  it is obvious that the first term at the right hand of (\ref{h02}) is a negative constant and the third term is always positive. Since $\beta+h_{0}(\eta) < 0$, by the results of parts $(a)$ and  $(d)$, $h_0''(\eta_3)>0$, then we have $h_{0}''(\eta_{3}) \int_{\eta_{3}}^{\eta}e^{-\int_{\eta_3}^x(h_{0}+\beta)dt}dx\geq h_{0}''(\eta_{3})(\eta-\eta_{3})$. Hence, $h_{0}'(\eta)$ is sufficiently large as $\eta$ is close to $2/\epsilon$, then $h_0(2/\epsilon)$ can not be close to $0$, in  contradiction to a boundary layer function.
\end{proof}

Although we can prove $\beta\geq0$, it is still difficult to determine $\beta$ analytically. We thus determine $\beta$ numerically. Gradually increasing $R$ and comparing the type $III$ numerical solution of (\ref{odef}) and (\ref{bc}) for a given boundary value $a$ with the solution of the reduced problem as in  expression (\ref{f02}), we can numerically estimate $\beta$. The results are summarised in \cref{tab:a-beta}. Then, it is obvious that $b=\frac{\cos^{-1}\beta}{2}$.  Then, we can  solve the boundary layer equation (\ref{h0}) subject to (\ref{bch02}) numerically.  The numerical results for $h_0(\eta)$ show that $h_0(\eta)\rightarrow0$ as $\eta\rightarrow2/\epsilon$.  Finally, the asymptotic solution up to $O(\epsilon)$ is $f(y)=f_0(y)+h_0(\eta)+O(\epsilon)$. This will be compared with the  numerical solution in  next section.

\begin{table}[tbhp]
{\footnotesize
 \caption{The numerical results of $\beta$ at  different given boundary values $a$ for $R=1500$. }\label{tab:a-beta}
\begin{center}
  \begin{tabular}{|c|l|l|l|l|l|l|l|} \hline
   $a$  &0.9  &0.8  &0.7  &0.6  &0.5  &0.4  &0.3 \\ \hline
   $\beta$&0.0889 &0.0783 &0.0672 &0.0551 &0.0417 &0.0264 &0.0079 \\ \hline
  \end{tabular}
\end{center}
}
\end{table}

\section {Comparison of numerical and asymptotic solutions}
\label{compa}
Numerical solutions for (\ref{ode1}) and (\ref{bc}) can be readily obtained by MATLAB boundary value problem solver bvp4c.  Comparison of the asymptotic solution and numerical solution will be shown in the following  Tables. 

For the type $I$ solution, we will make comparison between numerical and asymptotic solution for $f'(y)$ so as to see the accuracy of the type $I$ asymptotic solution constructed in $(\ref{type1})$. From Table \ref{tab:typeI}, it can be seen that the asymptotic solution is matched well with  the numerical solution.

\begin{table}[tbhp]
{\footnotesize
 \caption{Comparison between numerical and asymptotic results for $f'(y)$ at  $a=0.8$ with $R=100$,  $R=200$ and $R=300$.}\label{tab:typeI}
\begin{center}
  \begin{tabular}{|r|c|c|c|c|c|c|} \hline
  $f'(y)$  &\multicolumn{2}{c}{$R=100$} &\multicolumn{2}{c}{$R=200$} &\multicolumn{2}{|c|}{$R=300$}\\ \hline
$y$  & Numeric &Asymptotic  & Numeric & Asymptotic & Numeric & Asymptotic\\ \hline
-1.0    &0.0000 &0.0000               &0.0000&0.0000               &0.0000&0.0000\\ \hline
-0.8    &0.1780 &0.1781               &0.1771 &0.1771              &0.1768&0.1768\\ \hline
-0.6    &0.1603&0.1603               &0.1593 &0.1593               &0.1590&0.1590\\ \hline
-0.4    &0.1417&0.1417               &0.1409&0.1409               &0.1406&0.1406\\ \hline
-0.2    &0.1226&0.1226               &0.1219&0.1219               &0.1217&0.1217\\ \hline
-0.0    &0.1031&0.1030               &0.1023&0.1023               &0.1022&0.1022\\ \hline
0.2     &0.0829&0.0829               &0.0824&0.0824               &0.0823&0.0823\\ \hline
0.4     &0.0625& 0.0625              &0.0621&0.0621               &0.0620&0.0620\\ \hline
0.6     &0.0418&0.0418               &0.0416&0.0416               &0.0415&0.0415\\ \hline
0.8     &0.0209&0.0209               &0.0208&0.0208               &0.0208&0.0208\\ \hline
1.0     &0.0000&0.0000               &0.0000&0.0000               &0.0000&0.0000\\ \hline  
\end{tabular}
\end{center}
}
\end{table}

For the type $II$ solution,  since the turning points  $y_1=-1+\Delta_{1}$ and $y_2=1-\Delta_{2}$ are unknown a priori, getting the values of them is very important and difficult. We will contrast  numerical and asymptotic results for the turning points. The asymptotic results of $\Delta_{1}$ and $\Delta_{2}$ are from (\ref{Delta1}) and (\ref{Delta2}).   From Table \ref{tab:typeII}, it can be seen   that the error between the numerical and asymptotic results of the turning points is decreasing with the increasing  $R$ and that $|\Delta_{1}|$ and $|\Delta_{2}|$ get smaller and smaller as $R$ increases. These verify  our constructing process of the type  $II$ asymptotic solution in previous section.

\begin{table}[tbhp]
{\footnotesize
\caption{Comparison between numerical and asymptotic results for the turning points  $y_1=-1+\Delta_{1}$  and  $y_2=1-\Delta_{2}$ at  $a=0.8$. } \label{tab:typeII}
\begin{center}
\begin{tabular}{|c|c|c|c|c|}\hline
{\multirow{2}{*}{$R$}}    &\multicolumn{2}{|c|}{$y_{1}=-1+\Delta_{1}$}   &\multicolumn{2}{|c|}{$y_{2}=1-\Delta_{2}$}\\
\cline{2-5} 
              & Numeric  & Asymptotic   & Numeric  & Asymptotic  \\ \hline
100               &-0.7449      &-0.7315          &0.5457        &0.4728\\ \hline
200               &-0.8263      &-0.8227          &0.6753        &0.6519\\ \hline
400               &-0.8914      &-0.8921          &0.7868        &0.7959\\ \hline
600               &-0.9203      &-0.9202          &0.8483        &0.8434\\ \hline
800               &-0.9363      &-0.9363          &0.8783        &0.8750\\ \hline
\end{tabular}
\end{center}
}
\end{table}

For the type $III$ solution, we will compare the  numerical solution with the type $III$ asymptotic solution. Because of the complexity of the boundary layer problem (\ref{h0}) and  (\ref{bch02}), we compute the asymptotic solution $f(y)=f_0(y)+h_0(\eta)+O(\epsilon)$ in the following way:  $f_0(y)$  is obtained from (\ref{f02}) and $\beta$ or $b$ is estimated from numerical solution of (\ref{odef}) and (\ref{bc}), and $h_{0}$ is obtained numerically based on solving (\ref{h0}) and (\ref{bch02}). The Table \ref{tab:typeIII} shows that the asymptotic solution matches well with the numerical solution for large Reynolds numbers. 

\begin{table}[tbhp]
{\footnotesize
\caption{Comparison between numerical and asymptotic solutions for $f(y)$ at $R=800$ with $a=0.652$,   $a=0.748$ and $a=0.876$.} \label{tab:typeIII}
\begin{center}
\begin{tabular}{|r|c|c|c|c|c|c|} \hline
$f(y)$  &\multicolumn{2}{|c|}{$a=0.652$} &\multicolumn{2}{|c|}{$a=0.748$} &\multicolumn{2}{|c|}{$a=0.876$}\\ \hline
$y$  & Numeric &Asymptotic  & Numeric & Asymptotic & Numeric & Asymptotic\\ \hline
-1.0   		&0.6520&0.6520		&0.7480&0.7480	&0.8760  &0.8760\\ \hline
-0.6    		&0.3489&0.3462	        &0.3590&0.3516	&0.3711  &0.3586\\ \hline
-0.4    		&0.4866&0.4838	        &0.4948&0.4880       &0.5047  &0.4933 \\ \hline
-0.2    		&0.6131&0.6103	        &0.6194&0.6133	&0.6271  &0.6169\\ \hline
 0.0     		&0.7255&0.7228	        &0.7301&0.7246	&0.7356  &0.7267 \\ \hline
 0.2    		&0.8212&0.8187	        &0.8242&0.8196       &0.8279  &0.8204\\ \hline
 0.4    		&0.8981&0.8959	        &0.8998&0.8960       &0.9019  &0.8960\\ \hline
 0.6    		&0.9543&0.9526	        &0.9550&0.9523       &0.9560  &0.9518 \\ \hline
 0.8    		&0.9885&0.9875	        &0.9887&0.9872       &0.9889  &0.9867\\ \hline
 1.0      		&1.0000&1.0000	       &1.0000&1.0000	&1.0000  &1.0000\\ \hline
\end{tabular}
\end{center}
}
\end{table}

\section{Conclusion}
\label{sec:conclu}

In  this article, we have considered the multiplicity and asymptotics  of similarity solutions for laminar flows in a porous channel with different permeabilities, in particular,  flows permeating from upper wall of the porous channel and exiting from the lower wall. We rigorously prove that there exist three similarity solutions designated as type $I$, type $II$ and type $III$ solutions, and then numerically show that three solutions exist for $R>14.10$. Meanwhile, the asymptotic solution for each of the three types of similarity solutions is constructed for the most  interesting and challenging high Reynolds number case and is also verified numerically. For the type $I$ solution, its streamwise velocity has  an exponentially rapid decay. For the type $II$ solution, there are two turning points and its streamwise velocity also has an exponentially rapid decay. For the type $III$ solution, there exists  an exponentially rapid change not only for its streamwise velocity (decay and then increase)  but also for its transverse velocity (decay).  The reversal flow occurs for both type $II$ and type $III$ solutions.

\section*{Acknowledgments}
The authors would like to thank  Professor  Martin Stynes for many discussions and  suggestions.

\bibliographystyle{siamplain}
\bibliography{references}
\end{document}

%% file: ex_shared.tex
\usepackage{lipsum}
\usepackage{amsfonts}
\usepackage{graphicx}
\usepackage{epstopdf}
\usepackage{algorithmic}
\usepackage{subfig}
\usepackage{multirow}

\ifpdf
  \DeclareGraphicsExtensions{.eps,.pdf,.png,.jpg}
\else
  \DeclareGraphicsExtensions{.eps}
\fi

\newsiamremark{remark}{Remark}
\newsiamremark{hypothesis}{Hypothesis}
\crefname{hypothesis}{Hypothesis}{Hypotheses}
\newsiamthm{claim}{Claim}

\headers{Multiple solutions and their asymptotics for laminar flow}{H. X. Guo,  C. F. Gui, P. Lin and M. F. Zhao}

\title{Multiple solutions and their asymptotics for laminar flows through a porous channel with different permeabilities\thanks{
\funding{This work is funded by the National Natural Science Foundation of China (No. 91430106, No.11771040, No.11371128), the Fundamental Research Funds for the Central Universities (No. 06500073) and by NSF grant DMS-1601885.}}}

\author {Hongxia Guo\thanks{Department  of Applied Mathematics, University of Science and Technology, Beijing, 100083, China 
  (\email{g1214619441@163.com}).}
  \and Changfeng Gui\thanks{Department of Mathematics, University of Texas at San Antonio, San Antonio, TX 78249, USA; Institute of Mathematics, Hunan University, Changsha, 410082, China
  (\email{changfeng.gui@utsa.edu}).}
\and Ping Lin\thanks{Division of Mathematics, University of Dundee, Dundee,  DD1 4HN, United Kingdom;  Department  of Applied Mathematics, University of Science and Technology, Beijing, 100083, China 
 (\email{plin@maths.dundee.ac.uk}).}
\and Mingfeng Zhao\thanks{Center for Applied Mathematics, Tianjin University, Tianjin, 300072, China 
  (\email{mingfeng.zhao@tju.edu.cn}).}}

\usepackage{amsopn}


%% file: ex_article_3_.bbl
\begin{thebibliography}{10}

\bibitem{berman1953laminar}
{\sc A.~S. Berman}, {\em Laminar flow in channels with porous walls}, Journal
  of Applied physics, 24 (1953), pp.~1232--1235.

\bibitem{brady1981steady}
{\sc J.~Brady and A.~Acrivos}, {\em Steady flow in a channel or tube with an
  accelerating surface velocity. an exact solution to the navier-stokes
  equations with reverse flow}, Journal of Fluid Mechanics, 112 (1981),
  pp.~127--150.

\bibitem{chellam2006effect}
{\sc S.~Chellam and M.~Liu}, {\em Effect of slip on existence, uniqueness, and
  behavior of similarity solutions for steady incompressible laminar flow in
  porous tubes and channels}, Physics of Fluids, 18 (2006), pp.~083601--10.

\bibitem{cox1991analysis}
{\sc S.~M. Cox}, {\em Analysis of steady flow in a channel with one porous
  wall, or with accelerating walls}, SIAM Journal on Applied Mathematics, 51
  (1991), pp.~429--438.

\bibitem{cox1991two}
{\sc S.~M. Cox}, {\em Two-dimensional flow of a viscous fluid in a channel with
  porous walls}, Journal of Fluid Mechanics, 227 (1991), pp.~1--33.

\bibitem{cox1997asymptotic}
{\sc S.~M. Cox and A.~C. King}, {\em On the asymptotic solution of a
  high--order nonlinear ordinary differential equation}, in Proceedings of the
  Royal Society of London A: Mathematical, Physical and Engineering Sciences,
  vol.~453, The Royal Society, 1997, pp.~711--728.

\bibitem{hongxia}
{\sc H.~Guo, P.~Lin, and L.~Li}, {\em Asymptotic solutions for the asymmetric
  flow in a channel with porous retractable walls under a transverse magnetic
  field}, Applied Mathematics and Mechanics (English Edition), submitted.

\bibitem{hastings1992boundary}
{\sc S.~Hastings, C.~Lu, and A.~MacGillivray}, {\em A boundary value problem
  with multiple solutions from the theory of laminar flow}, SIAM Journal on
  Mathematical Analysis, 23 (1992), pp.~201--208.

\bibitem{lu1997asymptotic}
{\sc C.~Lu}, {\em On the asymptotic solution of laminar channel flow with large
  suction}, SIAM Journal on Mathematical Analysis, 28 (1997), pp.~1113--1134.

\bibitem{lu1999matched}
{\sc C.~Lu}, {\em On matched asymptotic analysis for laminar channel flow with
  a turning point}, Electronic Journal of Differential Equations
  (EJDE)[electronic only], 1999 (1999), pp.~109--118.

\bibitem{lu1999uniqueness}
{\sc C.~Lu}, {\em On the uniqueness of laminar channel flow with injection},
  Applicable Analysis, 73 (1999), pp.~497--505.

\bibitem{lu1992asymptotic}
{\sc C.~Lu, A.~D. MacGillivray, and S.~P. Hastings}, {\em Asymptotic behaviour
  of solutions of a similarity equation for laminar flows in channels with
  porous walls}, IMA Journal of Applied Mathematics, 49 (1992), pp.~139--162.

\bibitem{macgillivray1994asymptotic}
{\sc A.~D. MacGillivray and C.~Lu}, {\em Asymptotic solution of a laminar flow
  in a porous channel with large suction: A nonlinear turning point problem},
  Methods and Applications of Analysis, 1 (1994), pp.~229--248.

\bibitem{majdalani2003moderate}
{\sc J.~Majdalani and C.~Zhou}, {\em Moderate-to-large injection and suction
  driven channel flows with expanding or contracting walls}, ZAMM-Journal of
  Applied Mathematics and Mechanics/Zeitschrift f{\"u}r Angewandte Mathematik
  und Mechanik, 83 (2003), pp.~181--196.

\bibitem{proudman1960example}
{\sc I.~Proudman}, {\em An example of steady laminar flow at large reynolds
  number}, Journal of Fluid Mechanics, 9 (1960), pp.~593--602.

\bibitem{raithby1971laminar}
{\sc G.~Raithby}, {\em Laminar heat transfer in the thermal entrance region of
  circular tubes and two-dimensional rectangular ducts with wall suction and
  injection}, International Journal of Heat and Mass Transfer, 14 (1971),
  pp.~223--243.

\bibitem{robinson1976existence}
{\sc W.~Robinson}, {\em The existence of multiple solutions for the laminar
  flow in a uniformly porous channel with suction at both walls}, Journal of
  Engineering Mathematics, 10 (1976), pp.~23--40.

\bibitem{sellars1955laminar}
{\sc J.~R. Sellars}, {\em Laminar flow in channels with porous walls at high
  suction reynolds numbers}, Journal of Applied Physics, 26 (1955),
  pp.~489--490.

\bibitem{sherstha1967singular}
{\sc G.~Sherstha}, {\em Singular perturbation problems of laminar flow through
  channels in a uniformly porous channel in the presence of a transverse
  magnetic field}, The Quarterly Journal of Mechanics and Applied Mathematics,
  20 (1967), pp.~233--246.

\bibitem{shih1987existence}
{\sc K.-G. Shih}, {\em On the existence of solutions of an equation arising in
  the theory of laminar flow in a uniformly porous channel with injection},
  SIAM Journal on Applied Mathematics, 47 (1987), pp.~526--533.

\bibitem{shrestha1968laminar}
{\sc G.~Shrestha and R.~Terrill}, {\em Laminar flow with large injection
  through parallel and uniformly porous walls of different permearility}, The
  Quarterly Journal of Mechanics and Applied Mathematics, 21 (1968),
  pp.~413--432.

\bibitem{skalak1978nonunique}
{\sc F.~M. Skalak and C.~Y. Wang}, {\em On the nonunique solutions of laminar
  flow through a porous tube or channel}, SIAM Journal on Applied Mathematics,
  34 (1978), pp.~535--544.

\bibitem{terrill1965laminarinjection}
{\sc R.~Terrill}, {\em Laminar flow in a uniformly porous channel with large
  injection}, The Aeronautical Quarterly, 16 (1965), pp.~323--332.

\bibitem{terrill1964laminar}
{\sc R.~Terrill and G.~Shrestha}, {\em Laminar flow through channels with
  porous walls and with an applied transverse magnetic field}, Applied
  Scientific Research, Section B, 11 (1964), pp.~134--144.

\bibitem{terrill1965laminar}
{\sc R.~Terrill and G.~Shrestha}, {\em Laminar flow in a uniformly porous
  channel with an applied transverse magnetic field}, Applied Scientific
  Research, Section B, 12 (1965), pp.~203--211.

\bibitem{terrill1966laminar}
{\sc R.~Terrill and G.~Shrestha}, {\em Laminar flow through a channel with
  uniformly porous walls of different permeability}, Applied Scientific
  Research, 15 (1966), pp.~440--468.

\bibitem{terrill1965laminardiff}
{\sc R.~M. Terrill and G.~M. Shrestha}, {\em Laminar flow through parallel and
  uniformly porous walls of different permeability}, Zeitschrift f{\"u}r
  Angewandte Mathematik und Physik (ZAMP), 16 (1965), pp.~470--482.

\bibitem{uchida1977unsteady}
{\sc S.~Uchida and H.~Aoki}, {\em Unsteady flows in a semi-infinite contracting
  or expanding pipe}, Journal of Fluid Mechanics, 82 (1977), pp.~371--387.

\bibitem{watson1991laminar}
{\sc P.~Watson, W.~Banks, M.~Zaturska, and P.~Drazin}, {\em Laminar channel
  flow driven by accelerating walls}, European Journal of Applied Mathematics,
  2 (1991), pp.~359--385.

\bibitem{yuan1956further}
{\sc S.~Yuan}, {\em Further investigation of laminar flow in channels with
  porous walls}, Journal of Applied Physics, 27 (1956), pp.~267--269.

\bibitem{zaturska1988flow}
{\sc M.~Zaturska, P.~Drazin, and W.~Banks}, {\em On the flow of a viscous fluid
  driven along a channel by suction at porous walls}, Fluid Dynamics Research,
  4 (1988), pp.~151--178.

\end{thebibliography}
